\documentclass[leqno,12pt]{amsart}
\usepackage{amsmath,amssymb}
\usepackage{amsthm}
\usepackage[utf8]{inputenc}
\usepackage[T1]{fontenc}

\usepackage{amssymb}
\usepackage{amsthm}
\usepackage{amsfonts}
\usepackage{amsmath}
\usepackage{anysize}
\usepackage{hyperref}
\usepackage{bbm}
\usepackage{xcolor}
\usepackage{verbatim}
\usepackage{epsfig}  		
\usepackage{epic,eepic}     
\usepackage{tikz}
\usepackage{tikz-cd}
\usetikzlibrary{matrix, calc, arrows}

\usepackage{geometry} 
\def\i{\sqrt{-1}}

\newtheorem*{theorem*}{Theorem}

\newtheorem{theorem}{Theorem}[section]

\newtheorem*{notation*}{Notation}

\newtheorem{lemma}[theorem]{Lemma}
\newtheorem{corollary}[theorem]{Corollary}
\newtheorem{remark}[theorem]{Remark}
\newtheorem{example}[theorem]{Example}
\newtheorem{proposition}[theorem]{Proposition}

\newtheorem{conj}[theorem]{Conjecture}
\newtheorem{definition}[theorem]{Definition}

\title[Wall-chamber decompositions for gMA equations]{Wall-chamber decompositions for 
generalised Monge-Amp\`ere equations}

\author{Sohaib Khalid}
\author{Zakarias Sjöström Dyrefelt}

\address{Sohaib Khalid \\ Department of Mathematics and Mathematical Statistics, Umeå University, MIT-huset, Linnaeus väg, 907 36 Umeå, Sweden.}
\email{sohaib.khalid@umu.se}

\address{Zakarias Sjöström Dyrefelt \\ Department of Mathematics, Aarhus University, Ny Munkegade 118, 8000, Aarhus C, Denmark. }
\email{dyrefelt@math.au.dk}


\begin{document}

\maketitle

\begin{abstract}
Generalised Monge-Amp\`ere equations form a large class of PDE including Donaldson's J-equation, inverse Hessian equations, some supercritical deformed Hermitian-Yang-Mills equations, and some Z-critical equations. Solvability of these equations is characterised by numerical criteria involving intersection numbers over all subvarieties, and in this paper, we aim to characterise algebraically what happens when these nonlinear Nakai-Moishezon type criteria
fail. As a main result, we show that under mild positivity assumptions,
there is a finite number of subvarieties violating the Nakai type criterion, and such subvarieties are rigid in a suitable sense. This gives first effective solvability criteria for these families of PDE, thus improving on work of Gao Chen, Datar-Pingali, Song and
Fang-Ma, and provides first existence results in higher dimensions of 
a natural PDE analog of Bridgeland's locally finite wall-chamber decomposition.
\end{abstract}

\section{Introduction}

\noindent Let $(X,\omega)$ be a compact K\"ahler manifold, with $\alpha := [\omega]\in H^{1,1}(X,\mathbb{R})$ its associated K\"ahler class. Fix moreover real closed differential $(p,p)$-forms $\mathbf{\Theta}_p$ and consider 
$$
\mathbf{\Theta} = \sum_{p = 1}^n \mathbf{\Theta}_p.
$$
The generalised Monge-Amp\`ere equation, introduced by Pingali \cite{Pingali}, seeks a K\"ahler form $\omega_{\varphi} := \omega + dd^c\varphi \in \alpha$ such that
\begin{equation} \label{gMA eq}
(\mathbf{\Theta} \wedge \exp(\omega_{\varphi}))^{[n,n]} = \kappa \exp(\omega_{\varphi})^{[n,n]},
\end{equation}
where $\alpha^{[n,n]}$ denotes the $(n,n)$-part of a differential form $\alpha$, and $\kappa$ is the only possible cohomological constant.
Note especially that if \(\Theta=\kappa\theta^k\) for a Kähler form \(\theta\), then \eqref{gMA eq} is an inverse Hessian equation, and if $\kappa = 1$ this is the J-equation. We refer to \cite{FangMa, DatarPingali} for details on this broad formalism, including how special cases of interest (e.g. J-equation, inverse Hessian equations, dHYM equation) can be interpreted within this framework (\cite{FangLaiMa, DatarPingali, FangMa}).

In one of the most studied special cases of Donaldson's J-equation \cite{DonaldsonJobservation}, seminal work has shown that the associated J-flow 
converges away from a union of subvarieties given roughly by the singular part of the Siu decomposition of a certain form. Moreover, recent work of \cite{DatarMeteSong} paints an intriguing conjectural picture where the J-equation, despite apparent obstructions coming from \emph{instability}, can be solved \emph{in the weak sense of currents} away from certain subvarieties. In each of these cases, these exceptional subvarieties remain to be well understood.

From a related point of view, influential works on the J-equation \cite{GaoChen}, the strongly related deformed Hermitian Yang-Mills (dHYM) equation, as well as the generalised Monge-Amp\`re equations \cite{DatarPingali, FangMa} have shown that existence of a solution to these equations can be characterised by an algebro-geometric numerical criterion that involves checking the positivity of certain intersection numbers along all subvarieties $V \subset X$, reminiscent of other notions of stability arising in complex differential and algebraic geometry. Mirroring the theory of K-stability and the Yau-Tian-Donaldson conjecture, there also exists a characterisation by testing on all \emph{test configurations} (see \cite{LejmiGabor}). More recently, Reboulet \cite{Reboulet} has also given a different numerical criterion in the spirit of the Hilbert-Mumford criterion from geometric invariant theory. In each of these cases, it is natural to expect that it is in fact enough to test for a finite number of test configurations (resp. subvarieties), related to the previously mentioned locus where a suitable geometric flow blows up. 

Finally, a recent idea is that of a differential geometric analog to the celebrated wall-chamber decompositions exhibited in Bridgeland stability. 
More precisely, the existence of a (locally) finite number of subvarieties characterising stability via e.g. the criteria of Gao Chen \cite{GaoChen}, would lead to (locally) finite wall-chamber decompositions; a question where in higher dimension the main challenge includes understanding the interaction between subvarieties of different (co-)dimension, and their respective roles in the theory. In this paper, we set out to answer this question for the aforementioned \emph{generalised Monge-Amp\`ere (gMA) equations}, and more generally, the Z-critical equation (see \cite{DMS}), which is conjectured to be the correct analog to Bridgeland stability on the side of PDE. Our results include 1) A comprehensive analysis of the J-equation on threefolds, providing the first natural higher-dimensional wall-chamber decompositions and effective testing criteria beyond the case of surfaces (cf. \cite{SohaibDyrefelt}). 2) Extensions of these results to a broad class of gMA equations and Z-critical equations in arbitrary dimensions under suitable positivity conditions. 3) Explicit constructions of non-trivial locally finite wall-chamber decompositions in any dimension, for gMA equations.

Alongside direct implications for the study of geometric PDE, the core idea of the paper is to consider what happens algebraically when Nakai-Moishezon type conditions fail (cf. also the Demailly-Paun theorem \cite{DemaillyPaun}). We thus address (special cases of) the following basic and natural question. 
Suppose $\Omega$ denotes some specific PDE which cannot be solved. What are the possible subvarieties $V$ that violate the Nakai-Moishezon type numerical criterion for solvability, that is, for which subvarieties $V$ (which we shall call \emph{destabilising subvarieties}) do we have 
    \[
    \int_V \vartheta_p(\Omega) \leq 0?
    \]
Here $\vartheta_p(\Omega)$ are the degree $(p,p)$ cohomology classes coming from the Nakai-Moishezon criterion associated to $\Omega$. More precisely:
\begin{enumerate}
    \item What constraints, if any, are there on the geometry of such $V$?
    \item What is the cardinality of the set of such $V$ if we keep $\Omega$ fixed?
    \item If $\Omega_b, b\in \mathcal B$ is a continuous family of PDEs parametrised by $\mathcal B$, what can we say about the set of subvarieties that occur as destabilising subvarieties for $\Omega_{b_0}$ for some $b_0 \in \mathcal B$? 
\end{enumerate}
\subsection{Statement of main results}
Our techniques can be applied in a fairly broad setting, namely the class of gMA equations of the form \eqref{gMA eq} with \[\mathbf\Theta = \sum_{k=0}^{n-1} c_k \theta^k + f\theta^n\]
where $\theta$ is a fixed Kähler form, $c_k$ are real constants and $f$ is a smooth function. These are by definition subject to a standard positivity constraint, which in particular includes that each $c_k$ be nonnegative. (See \cite[eq. (1.2)]{DatarPingali} for a precise statement.) 

The first key observation here is that \emph{any} gMA equation is \emph{factorisable} in a sense we make precise in this paper (see Definition \ref{def:factorisablegMAeq}). Concretely, this means that we can `factorise' the cohomology classes occurring in the numerical criterion for solvability, so that the following holds: \eqref{gMA eq} is solvable if and only if for each $1\leq p \leq n-1$, and each proper irreducible subvariety $V$ of $X$ of dimension $p$, we have
\[
\int_V \tau_p(\alpha,\mathbf\Theta) \cdot \rho(\alpha,\Theta) > 0.
\]
Here the $\tau_p = \tau_p(\alpha,\mathbf\Theta)$ (for $p=1,2,\dots,n-1$) are certain $(1,1)$-cohomology classes depending only on the gMA equation under consideration, that is, on the data $(X,\alpha,\mathbf\Theta)$, which we call the \emph{factor classes}, and $\rho_p = \rho_p(\alpha,\mathbf\Theta)$ are $(p-1,p-1)$ classes that can be written as positive linear combinations of wedge products of Kähler classes. Each destabilising subvariety of dimension $p$ then lies in the non-Kähler locus of $\tau_p$. We refer to Lemma \ref{lem:gMAfactorisable} for details on how to obtain $\tau_p$ from $(X,\alpha,\mathbf\Theta)$, where we also prove that all gMA equations are in fact factorisable. 

From this observation, and by virtue of the numerical criterion established by the work of Datar-Pingali \cite{DatarPingali} and Fang-Ma \cite{FangMa}, we obtain the following main result, providing a general tool for effective testing for solvability and an analog to Bridgeland's wall-chamber decompositions:  

\begin{theorem} \label{main thm gMA intro}
Let $X$ be a compact Kähler manifold of dimension $n$ and let $\alpha,\beta$ be Kähler classes on $X$ and $\theta \in \beta$ a fixed Kähler form. Suppose that 
    \[
    \mathbf\Theta = \sum_{k=1}^{n-1}c_k \theta^k + f\theta^n
    \]
    with $c_k\geq 0$, $f\geq 0$ not all zero satisfies the condition \cite[(1.2)]{DatarPingali}. 
\begin{enumerate}
    \item \emph{\textbf{(Factorisability)}} The generalised Monge-Amp\`ere equation associated to $(X,\alpha,\mathbf\Theta)$ is factorisable. Denote by $\tau_p$ the associated factor classes.
    \item \emph{\textbf{(Effective testing)}} Suppose that  $\tau_p$ are $(p+1)$-modified Kähler for each $p=1,\dots, n-1$. Then, there is a finite set of subvarieties $V_1,\dots,V_l$ such that 
    $$
    \int_{V_i} \exp(\alpha) \cdot (\kappa - [\mathbf\Theta]) > 0, \; \; i = 1,2,\dots,l
    $$
    implies existence of a solution to the gMA equation \eqref{gMA eq}.
    \item \emph{\textbf{(Chamber decomposition)}} Let $\mathcal S_+(X, \mathrm{gMA})(\alpha)$ be the set of all $\mathbf\Theta$ as above with $f = c_n$ a constant, endowed with the natural Euclidean topology. Let $\mathcal S_+(X,\mathrm{gMA})(\alpha)^{\mathrm{Stab}}$ denote the locus of those $\mathbf\Theta$ which define a gMA equation admitting a smooth solution. For every $\Gamma
    \in\partial \mathcal S_+(X, \mathrm{gMA})(\alpha)^{\mathrm{Stab}}$ belonging to the boundary there exists an open neighbourhood $U$ of $\Gamma$ in $\mathcal S_+(X, \mathrm{gMA})(\alpha)$ such that $U\cap \partial \mathcal \mathcal S_+(X, \mathrm{gMA})(\alpha)^{\mathrm{Stab}}$ is a finite union of walls cut out by equations of the form \[\int_V \exp(\alpha)\cdot (\kappa-[\mathbf\Theta]) = 0.\] 
\end{enumerate}
\end{theorem}
\begin{remark}
\emph{We can also show that all supercritical dHYM equations are factorisable when $\dim X \leq 3$ and some are factorisable for $\dim X = 4$; see Theorem \ref{thm:dHYM}.}
\end{remark}

\noindent We reiterate that Main Theorem \ref{main thm gMA intro} applies in particular to inverse Hessian equations, and a version of this result also holds for the supercritical deformed Hermitian Yang-Mills (dHYM) equation, see Section \ref{subsection supercritical dHYM} The close parallels with the chamber decompositions occurring in Bridgeland stability are elaborated on in Section \ref{Section chamber}.

\begin{remark}
\emph{(Z-critical equations) Another important class studied in recent literature are the Z-critical equations, introduced in \cite{DMS} as differential geometric analogues of special polynomial stability conditions \cite{Bayer} modeled on the theory of Bridgeland \cite{Bridgeland}. The above theorem applies to a subclass of these, given appropriate choice of stability datum. We explain a particular example of this more concretely, adopting the notation and conventions of \cite[Section 2.1]{DMS}. Let $L$ be a holomorphic line bundle on a compact Kähler manifold $X$ of dimension $n$. Let us consider unipotent cohomology classes of the form $U = \exp(A)$ where $A \in H^{1,1}(X,\mathbb R)$ is any cohomology class such that $A+c_1(L) \in \mathcal K_X$. Then, for any stability datum $\Omega = (\beta, \rho, \exp(A))$ where the stability vector $\rho=(\rho_0,\rho_1,\dots,\rho_n)$ satisfies 
\[
b_k := \frac{\operatorname{Im}\left(\overline{\operatorname{Z}_\Omega(L)}\rho_k\right)}{\operatorname{Im}\left(\overline{\operatorname{Z}_\Omega(L)}\rho_0\right)}<0, \quad k = 1, 2, \dots, n,
\]
we obtain a gMA equation satisfying the positivity conditions given in \cite[(1.2)]{DatarPingali} for the Kähler class $\alpha = A + c_1(L)$ and 
\[
\mathbf\Theta = \sum_{k=1}^n b_k \theta^k
\]
for any Kähler form $\theta\in\beta$. Thus, Theorem \ref{main thm gMA intro} applies to this class of Z-critical equations.}
\end{remark}

\begin{remark}
    \emph{More broadly, one should study Z-critical equations on higher rank vector bundles (or, indeed, arbitrary coherent sheaves) in order to understand the full force of the theory of stability conditions from a differential geometric standpoint. From the point of view of the present work, it would be desirable to understand the set of destabilising objects in higher rank. Here, the main difficulty is that destabilising objects may now come in two different forms. In addition to there being destabilising subvarieties, there might also be destabilising subsheaves of the associated holomorphic vector bundle. However, some of the same considerations might apply to this much more complicated setting. A special case of this setup has recently been considered by Keller-Scarpa \cite{KellerScarpa}. In their work, they term what we call stability as \emph{positivity} and reserve the term \emph{stability} for the higher rank condition involving subsheaves (or conditions that involve a mixture of both subsheaves and subvarieties). They propose general conjectures that relate positivity, stability and the existence of solutions of the relevant PDEs.}
\end{remark}

\subsection{Effective testing and wall-chamber decomposition for the J-equation}
We now take a closer look at the important special case of Donaldson's J-equation, in which we can improve further on our main theorem, and infer concrete applications and characterisations of the special subvarieties we study in the context of J-stability and stability thresholds. The first results in the direction of chamber decompositions were recently proven in \cite{SohaibDyrefelt}, showing that locally finite wall-chamber decompositions always exist on K\"ahler surfaces. In dimension $3$ and onwards the question is completely open, lacking even non-trivial examples confirming this conjectural picture. Compared with the surface case the higher dimensional picture is much more subtle, for the two main reasons that: a) intersection theory and positivity notions are much more well-behaved on surfaces b) in higher dimension there is a complex interaction between subvarieties of different codimensions. 

In the spirit of slope stability, we will use the following notation: if $X$ is a compact Kähler manifold and $\alpha,\beta$ are Kähler classes on $X$, set 
$$
\mu_{\alpha,\beta}(V) := (\dim V) \frac{\alpha^{\dim V - 1}\cdot \beta \cdot [V]}{\alpha^{\dim V}\cdot [V]}.
$$
We will abbreviate $\mu_{\alpha,\beta}(X) =: \mu_{\alpha,\beta}$. The triple $(X,\alpha,\beta)$ is called $J$-stable (respectively $J$-semistable) if for all proper irreducible analytic subvarieties $V$ of $X$, we have 
$
\mu_{\alpha,\beta}(V) < \mu_{\alpha,\beta} \quad \textrm{ (respectively } \mu_{\alpha,\beta}(V)\leq \mu_{\alpha,\beta}\textrm{)},
$
and $J$-unstable if it is not $J$-semistable. We will often drop the prefix and refer to these notions as simply stable or semistable respectively.
We will also make use of the \emph{stability threshold} studied in e.g. \cite{SD4}, given by
$$
\Delta^{\mathrm{pp}}_{\beta}(\alpha) := \inf_{V \subset X} \frac{1}{\dim X-\dim V}\left(\mu_{\alpha,\beta} - \mu_{\alpha,\beta}(V)\right),
$$

which is strictly positive precisely if the J-equation is solvable for $(X,\alpha,\beta)$.

As main objects of study we then consider the following sets of destabilising (resp. optimally destabilising) subvarieties for the $\mathrm{J}_{\alpha,\beta}$-equation: 
\begin{equation}
\mathrm{Dest}_{\alpha,\beta} := \{V \subset X \; \vert \; \mu_{\alpha,\beta} \leq \mu_{\alpha,\beta}(V) \}
\end{equation}
\begin{equation}
\mathrm{Dest}_{\alpha,\beta}^{\mathrm{opt}} := \left\{V \subset X \; \vert \; \frac{1}{\dim X-\dim V}\left(\mu_{\alpha,\beta} - \mu_{\alpha,\beta}(V)\right) = \Delta^{\mathrm{pp}}_{\beta}(\alpha) \leq 0 \right\} 
\end{equation}

\noindent It is natural to aim to understand elements of these sets, in particular under which conditions they are finite. 
We begin by stating a general version of our main result for the J-equation.

\begin{theorem} \label{main thm 1 intro}
Suppose $X$ is a compact K\"ahler manifold. Let $\alpha,\beta$ be K\"ahler classes on $X$. 
\begin{enumerate}
    \item (\cite{SohaibDyrefelt}) Suppose $\dim X = 2$. Then, there are only finitely many destabilising
curves, each of which is a curve of negative self-intersection.
    \item Suppose $\dim X = 3$ and $$\tau_2(\alpha,\beta) := \mu\alpha - 2\beta$$ is a big cohomology class on $X$. Then, there exists a proper analytic subset $V_{\alpha,\beta}$ such that each irreducible component of $V_{\alpha,\beta}$ is a destabilising subvariety and all destabilising subvarieties are contained in $V_{\alpha,\beta}$. If moreover
$(X, \alpha, \beta)$ is J-semistable, then there exist only finitely many
destabilising subvarieties. Each irreducible component $V$ of the set $V_{\alpha,\beta}$ is moreover rigid in
the following sense: 
\begin{itemize}
    \item every irreducible surface component is the unique effective analytic cycle representing its homology class.
    \item every irreducible curve component $C$ satisfies the following: for every irreducible surface $S$ containing $C$, either $C$ is an irreducible component of the singular locus of $S$ or (the strict transform of) $C$ is a curve of negative self-intersection in (any resolution of singularities of) $S$.
\end{itemize}
    \item Suppose $$\tau_p(\alpha,\beta) := \mu\alpha - p\beta$$ is a $(p+1)$-modified K\"ahler class for each $p = 1,2,\dots, \dim X -1$. If moreover $(X,\alpha,\beta)$ is J-semistable, then there exist only finitely
many destabilising subvarieties, and each destabilising divisor is the only effective analytic cycle
representing its homology class. 
\item Suppose $$(\mu - (n-p)\Delta^{\mathrm{pp}}_{\beta}(\alpha))\alpha - p\beta$$ is a $(p+1)$-modified K\"ahler class for each $p = 2,3,\dots, \dim X -1$. Then, there is a finite number of optimal destabilisers. 
\end{enumerate}
\end{theorem}

\begin{remark}
\emph{There are many cases of $(X,\alpha,\beta)$ for which $\Delta^\mathrm{pp}_{\beta}(\alpha)$ is naturally negative enough so that $(\mu - (n-p)\Delta^\mathrm{pp}_{\beta}(\alpha))\alpha - p\beta$ are all in fact $(p+1)$-modified K\"ahler. This ensures that the above statements are not void, and in fact we will show that they are applicable to large classes of examples, see Section \ref{Section example} and Theorem \ref{Cor blowup} below.}
\end{remark}

\begin{remark} \emph{The proof gives a finite list of subvarieties over which it is enough to test J-positivity (similarly, from the point of view of J-stability, it suffices to consider test configurations arising as the deformation to the normal cone for these finitely many $V_i$). Moreover, if $\dim X = 2$ or $3$, then any optimally destabilising subvariety is rigid in a suitable sense, see Section \ref{Section main results J}. More generally, the divisors occurring as potential optimal destabilisers can be recognised as precisely those with strictly positive Lelong number $\nu(\mu\alpha - (n-1)\beta, D) > 0$, i.e. the finitely many divisors in the negative part of the divisorial Boucksom-Zariski decomposition \cite{Boucksomthesis}.}
\end{remark}
\begin{remark}
\emph{
If the classes $\mu\alpha - p\beta$ are all $(p+1)$-modified K\"ahler, then, by openness of the modified K\"ahler cones, there is a strictly positive $\lambda > 0$ such that the $(\mu - (n-p)\lambda)\alpha - p\beta$ remain $(p+1)$-modified K\"ahler. This leads to a fascinating `gap statement' improving further on Theorem \ref{main thm 1 intro}, namely 
that only finitely many subvarieties satisfy
    $$
    \frac{1}{\dim X-\dim V}\left(\mu_{\alpha,\beta} - \mu_{\alpha,\beta}(V)\right) \leq \Delta^{\mathrm{pp}}_{\beta}(\alpha) + \lambda.
$$
}
\end{remark}
\begin{remark}
\emph{The above theorem holds under certain positivity assumptions on $\mu\alpha - p\beta$ being $(p+1)$-modified K\"ahler. While we do expect that these statements can be improved in general, some positivity hypotheses for these classes are necessary: 
Indeed, it suffices to consider the blowup of $\mathbb{P}^n$ at a point, 
where the exceptional divisor $E$ contains infinitely many $(n-2)$-dimensional linear hyperplanes that, depending on the classes $\alpha,\beta$ considered, may all be destabilising. However, as long as $\mu_{\alpha,\beta}\alpha - 2\beta$ is big (and $n=3$), the above theorem prohibits these lines from being destabilising at semistability, and one may verify by direct computation that as we pass from the stable to the unstable locus only $E$ itself will in fact destabilise. 
Interestingly, if we drop the hypothesis that $\mu_{\alpha,\beta}\alpha - (n-2)\beta$ is $(n-1)$-modified K\"ahler, there may be infinitely many destabilisers (see Example \ref{Subsection ex}). 
See also Example \ref{Subsection ex Cutkosky} and Section \ref{Subsection optimal dest Jstab}.}
\end{remark}

\noindent We now discuss several applications of the above theorem, that serve as key motivation for this paper. First, it provides an improvement to results of Gao Chen \cite{GaoChen} and Datar-Pingali \cite{DatarPingali}, Song \cite{Song}, and takes steps toward a more effective testing of solvability of the J-equation. That solvability implies stability in this context is known from \cite{CollinsGabor}, and with our contribution it adds up to the following strengthening of \cite[Main theorem]{GaoChen}:
\begin{theorem}
\label{Cor finite test intro}
\begin{enumerate}
    \item Suppose the hypotheses of Theorem \ref{main thm 1 intro} (2) hold. Then the J-equation is solvable on $(X,\alpha,\beta)$ if and only if 
    \[\mu_{\alpha,\beta} > \mu_{\alpha,\beta}(V)\] for $V$ belonging to the following finite list: $S_1,\dots, S_r, C_1, \dots, C_\ell$, where the $S_i$ are the irreducible surface components of the non-Kähler locus of $\mu_{\alpha,\beta} \alpha - 2 \beta$ and $C_j$ are the irreducible components in the negative part of the Zariski decomposition of $(\mu_{\alpha,\beta}\alpha-\beta)\vert_{S_i}$ for $i = 1,\dots, r$.
    \item Suppose the hypotheses of Theorem \ref{main thm 1 intro} (3) hold. Then the J-equation is solvable on $(X,\alpha,\beta)$ if and only if for every one of the finitely many $p$-dimensional irreducible components $V$ of the non-Kähler locus of $\tau_p$, and all $p = 1,2,\dots,n-1$, we have
$$
\mu_{\alpha,\beta} > \mu_{\alpha,\beta}(V).
$$
\end{enumerate}
\end{theorem}

\noindent Second, Theorem \ref{main thm 1 intro} provides a first result in the direction of wall-chamber decompositions for PDE, formulated as follows for the J-equation: 

\begin{theorem}\label{cor wall chamb dec intro}
Let $\mathcal{S}_+(X,J)$ be the set comprising pairs of K\"ahler classes $(\alpha, \beta)$ such
that $\mu_{\alpha,\beta}\alpha - p\beta$ is a $(p + 1)$-modified Kähler class for
$p = 1, . . . , dim X - 1$, and set
$$
\mathcal{S}_+(X,J)^{\mathrm{Stab}} := \{(\alpha, \beta) \in \mathcal{S}_+(X, J) : (X, \alpha, \beta) \; \textrm{is} \; \textrm{J-stable}.\}.
$$
Then, the boundary
$
\partial \mathcal{S}_+(X,J)^{\mathrm{Stab}}
$
is a locally finite union of smooth
submanifolds of codimension one in $\mathcal{S}_+(X, J)$. Each such submanifold is
cut out by a destabilising subvariety $V$, that is by a condition of the form
$$
\mu_{\alpha,\beta} - \mu_{\alpha,\beta}(V) = 0.
$$
\end{theorem}

\noindent Finally, we illustrate that the above positivity conditions in fact occur naturally in many situations, thanks to the following sufficient condition for a pair $(\alpha,\beta)$ to be an element of $\mathcal{P}_X$:

\begin{theorem}
Fix $X$ a compact K\"ahler manifold and $\alpha,\beta$ K\"ahler classes on $X$. Then, the conclusions of Theorems \ref{main thm 1 intro} and Corollaries \ref{Cor finite test intro} and \ref{cor wall chamb dec intro} above all hold, if there exists a nef class $\eta \in H^{1,1}(X,\mathbb{R}) \setminus \mathcal K_X$ such that  $\alpha = (1-r)\eta + r\beta$ for some $r \in (0,1)$, and $\eta$ is $3$-modified K\"ahler\footnote{Note that this automatically implies that $\eta$ is big, and thus $\eta^n > 0$.}.
\end{theorem}

\begin{remark}
\emph{Such manifolds do exist, in any dimension, thus giving first examples of wall-chamber decomposition phenomena for generalised Monge-Amp\`ere equations beyond the surface case. In particular, this is true for a family of K\"ahler $n$-folds $X=\mathbb P( \mathcal O_Y \oplus (L^\vee)^{\otimes a_1} \oplus \dots \oplus (L^\vee)^{\otimes a_{n-1}})$ studied by Wu \cite{Wu}, see Section \ref{Subsection ex Cutkosky} for notation and details. } 
\end{remark}

\noindent The above theorem is particularly useful if $X$ is a threefold, when the above condition becomes $\eta \in \mathcal{M}_3\mathcal{K}_X = \mathcal{B}_X$, the big cone of $X$; something which is often checkable in practice. As a consequence any threefold where $\mathcal{B}_X \neq \mathcal{K}_X$ carries many pairs of classes $(\alpha,\beta)$ such that $\mathrm{Dest}_{\alpha,\beta}$ is a finite set. 

More generally, if $\eta \in \mathcal{M}_k\mathcal{K}_X$ for some $k$, $3 \leq k \leq n$, then $\mathrm{Dest}_{\alpha,\beta}$ contains no subvarieties of dimension $k$ or higher.

The above in particular shows that our main theorems give non-trivial results on the blow-up of \emph{any} compact K\"ahler threefold, 
for all K\"ahler classes $\alpha$ in an open \emph{subset} of the K\"ahler cone, in the following sense.

\begin{theorem} \label{Cor blowup}
Let $X$ be any compact K\"ahler threefold, and let $\pi: Y \rightarrow X$ be the blow up of $X$ at a point $x \in X$. If $\alpha,\beta$ are K\"ahler classes on $X$, then consider the nef and big classes $\pi^*\alpha, \pi^*\beta$ on $Y$ and write $\alpha_{\epsilon} := \pi^*\alpha - \epsilon[E]$, and $\beta_{\epsilon} := \pi^*\beta - \epsilon[E]$. If moreover  $0 < \epsilon <<  \epsilon'$ are both small enough, then the set
$
\mathrm{Dest}_{\alpha_{\epsilon},\beta_{\epsilon'}}^{\mathrm{opt}}
$
of optimally destabilising subvarieties is finite, and any divisor component has strictly positive generic Lelong number 
$$
\nu((\mu - \Delta^\mathrm{pp}_{\beta_{\epsilon'}}(\alpha_{\epsilon}))\alpha_{\epsilon} - 2\beta_{\epsilon'},D) > 0
$$
If we have strict inequality in Lemma \ref{Lemma double ineq} then any optimal destabiliser is a divisor.
\end{theorem}

\subsection{Acknowledgements}
We are grateful to Ved Datar, Antonio Trusiani, Jacopo Stoppa, Sivaram Petchimuthu and Masafumi Hattori for helpful discussions. The authors would also like to thank the Isaac Newton Institute for Mathematical Sciences, Cambridge, for support and hospitality during the programme ``New equivariant methods in algebraic and differential geometry'' where work on this paper was undertaken. This work was supported by EPSRC grant no EP/R014604/1. The first named author acknowledges support from grant JCSMK24-0084 from the Kempe Foundation. The second named author was supported by an AIAS-COFUND II Marie Curie Grant, and Villum Young Investigator Grant, project no. 60786. We also thank the anonymous referee for several helpful suggestions.

\bigskip

\section{Background on positive cones and non-K\"ahler loci} \label{Section prelim}

\subsection{Currents and positive cones} 

Recall that given a pseudoeffective class $\tau$ on a compact Kähler manifold $(X,\omega)$, we can define the \emph{minimal multiplicity} $\nu(\tau,x)$ of $\tau$ at $x$ by \[
\nu(\tau,x):= \sup_{\varepsilon > 0} \nu(T_{\min,\varepsilon},x)
\]
where $T_{\min,\varepsilon}$ is a closed $(1,1)$ \emph{current of minimal singularities} in $\tau$ that satisfies $T_{\min,\varepsilon} \geq -\varepsilon\omega$. This definition is independent of the choice of the Kähler metric $\omega$. For an irreducible analytic subvariety $V$ of $X$, a pseudoeffective class $\tau$ on $X$ and a positive current $T$, we denote 
\[
\nu(\tau,V) := \inf_{x\in V} \nu(\tau,x), \quad \textrm{and } \quad \nu(T,V) :=\inf_{x\in V} \nu(T,V).
\]
See \cite[Section 2.8]{Boucksomthesis} for precise definitions. 

The following definition is due to \cite{Wu} and generalises the notion of a \emph{modified nef class} defined by Boucksom in \cite{Boucksomthesis}.

\begin{definition}[nef and Kähler in codimension $q$]  Let $X$ be a compact Kähler manifold of dimension $n$. A pseudoeffective $(1,1)$ class $\tau$ on $X$ is said to be \emph{nef in codimension $q$} or \emph{$(n-q)$-modified nef} if the minimal multiplicity $\nu(\tau,Z) = 0$ for any irreducible analytic subvariety of codimension $k \leq q$. These classes comprise a closed cone $\mathcal M^q\mathcal{N} = \mathcal M_{n-q}\mathcal N$, which we shall call the \emph{$(n-q)$-modified nef cone} and its interior $\mathcal M^q\mathcal K =\mathcal M_{n-q}\mathcal K$ the \emph{$(n-q)$-modified Kähler cone}. A class $\tau \in \mathcal M_p \mathcal K$ is called a \emph{$p$-modified Kähler class}.
\end{definition}
\begin{remark}
    We have an obvious inclusion of cones 
    \[
    \mathcal N = \mathcal M_0 \mathcal N \subseteq \mathcal M_1 \mathcal N \subseteq \ldots \subseteq \mathcal M_n \mathcal N = \mathcal E
    \]
    and similarly for $\mathcal M_i \mathcal K$. In particular, note that $\mathcal M_n \mathcal K = \mathcal B$ is the big cone.
    In fact, as a corollary of a result of Pa\u un (see Lemma 2 in \cite{Wu}), we also have $\mathcal M^{n-1}\mathcal N = \mathcal M_1 \mathcal N = \mathcal N$. In \cite{Boucksomthesis}, the cone $\mathcal M_{n-1}\mathcal N$ is called the \emph{modified nef cone}.
\end{remark}
Recall that the \emph{non-Kähler locus} $E_{nK}(\tau)$ of a big class $\tau$ is given 
\[
E_{nK}(\tau) := \bigcap_{\tilde T \in \tau} E_+(\tilde T)
\]
where the intersection ranges over all strictly positive closed currents $\tilde T$ representing $\tau$. Here $E_+(\tilde T)$ is the set comprising $x\in X$ such that $\nu(\tilde T,x)>0$. The set $E_{nK}(\tau)$ is always a closed analytic subset of $X$ whenever $\tau$ is a big class on $X$.

The following is a straightforward consequence of standard results and the above definitions.
\begin{lemma}\label{lem:EnK}
    Let $X$ be a compact Kähler manifold of dimension $n$ and let $\tau \in \mathcal M_n \mathcal K$ be a big class on $X$. Then, $\tau\in\mathcal M_p \mathcal K$ on $X$ if and only if every irreducible component of $E_{nK}(\tau)$ has dimension strictly less than $p$.
    \begin{proof}
       Fix a Kähler metric $\omega$ on $X$. First suppose $\tau \in \mathcal M_p\mathcal K$. Then, we can find $\varepsilon > 0$ so small that $\tau - 2\varepsilon [\omega] \in  \mathcal M_p \mathcal N$. Thus, for any $V\subseteq X$ of dimension at least $p$, we have by definition 
        \[
        \nu(\tau - 2\varepsilon[\omega],V) = 0.
        \]
        In particular, we have 
        \[
        \nu(T_{\min,\varepsilon},V)= 0
        \]
        where $T_{\min,\varepsilon}$ is a current of minimal singularities in $\tau - 2\varepsilon[\omega]$ satisfying $T_{\min,\varepsilon}\geq -\varepsilon\omega$. But then, the current
        \[
        T:= T_{\min,\varepsilon}+2\varepsilon\omega
        \]
        lies in the class $\tau$ and satisfies $T \geq \varepsilon \omega$. Thus, $T$ is a Kähler current in $\tau$ and satisfies $\nu(T,V) = 0$ for all $V\subseteq X$ of dimension at least $p$. Since $E_{nK}(\tau) \subseteq E_+(T)$, we see that no such $V$ can be contained in $E_{nK}(\tau)$. 
        
        Conversely, suppose that $E_{nK}(\tau)$ does not contain any irreducible analytic subset of dimension greater than or equal to $p$. Then, because $\tau$ is big, by \cite[Theorem 3.17(ii)]{Boucksomthesis} there exists a Kähler current $T\in \tau$ such that $E_+(T) = E_{nK}(\tau)$. Because $T$ is a Kähler current, we can find $\varepsilon > 0$ so that $T\geq 2\varepsilon\omega$ and $\tau - \varepsilon[\omega]$ is a big class. In particular, $\nu(T-\varepsilon\omega,V) = 0$ for all $V\subseteq X$ of dimension at least $p$. Let $T_{\min}$ be any current of minimal singularities in $\tau -\varepsilon[\omega]$ satisfying $T_{\min} \geq 0$. Then, by minimality, we have 
        \[
        \nu(T_{\min} ,x) \leq \nu(T-\varepsilon\omega,x)
        \]
        for all $x\in X$. Thus, $\nu(T_{\min},V) = 0$ for all $V$ of dimension at least $p$. Now, by \cite[Proposition 3.6(ii)]{Boucksomthesis}, we have $\nu(T_{\min},x) = \nu(\tau-\varepsilon[\omega],x)$, and so we obtain $\nu(\tau-\varepsilon[\omega],V) = 0$ whenever $V$ is of dimension at least $p$. Thus, $\tau - \varepsilon[\omega] \in \mathcal M_{p}\mathcal N$. Since this holds for an arbitrary Kähler class $[\omega]$, we obtain that $\tau \in \mathcal M_p \mathcal K$ is a $p$-modified Kähler class.
    \end{proof}
\end{lemma}

Throughout this paper, we shall use Lemma \ref{lem:EnK} in conjunction with the following straightforward consequence of results of Boucksom \cite{Boucksomthesis} and Demailly \cite{Demailly}. 

\begin{lemma}\label{lem:main lemma}
    Let $\tau$ be a big cohomology class on a compact Kähler manifold $X$, and $\omega_1, \omega_2,\ldots,\omega_{p-1}$ be smooth Kähler forms on $X$. Then, for any $p$-dimensional subvariety $V$ satisfying $V\not\subseteq E_{nK}(\tau)$, we have 
    $$
    \int_V \tau\cdot[\omega_1]\cdot[\omega_2]\cdot\ldots\cdot [\omega_{p-1}] > 0.
    $$
\end{lemma}
\begin{proof}
    By a result of Boucksom \cite[3.17 (ii)]{Boucksomthesis}, there exists a Kähler current with analytic singularities $T\in\tau$ such that $E_+(T)=E_{nK}(\tau)$. By rescaling $\omega_1$ if necessary, we may assume that $T\geq \omega_1$. By Demailly's regularisation result \cite[Main Theorem 1.1]{Demailly} there exist smooth forms $\theta_{k}$ representing $\tau$ and continuous functions $\lambda_{k}:X\to \mathbb R$ such that $\theta_{k}$ converge weakly to $T$, $\theta_{k}+\lambda_{k}\omega_1 \geq \omega_1$ in the sense of currents and $\lambda_{k}(x)$ decrease to $\nu(T,x)$. Since $X$ is compact, we can find smooth functions $\rho_{k}:X\to \mathbb R$ such that for all $x\in X$ we have $$
    0 \leq \rho_{k}(x)-\lambda_{k}(x) \leq 2^{-k}.
    $$
    Then, we obtain
    $$
    \theta_{k}+\rho_{k}\omega_1 \geq (1+2^{-k})\omega_1.
    $$ From this, it follows that
    \[
    (\theta_k+\rho_k\omega_1)\wedge \omega_1\wedge\ldots\wedge\omega_{p-1}\geq(1+2^{-k})\omega_1^2\wedge\ldots\wedge\omega_{p-1}
    \] as smooth measures on the regular part of $V$. Thus, we have 
    \begin{equation*}
        \int_V \tau \cdot [\omega_1]\cdot\ldots\cdot[\omega_{p-1}] + \int_V \rho_k \omega_1^2\wedge\omega_2\wedge\ldots\wedge\omega_{p-1} \geq (1+2^{-k}) \int_V \omega_1^2\wedge\omega_2\wedge\ldots\wedge\omega_{p-1}.
    \end{equation*}
    Now $\rho_{k}(x)$ also converges to $\nu(T,x)$, and thus the sequence $\rho_k$ of smooth functions converges to zero almost everywhere on the regular part of $V$ with respect to the measure $\omega_1^2\wedge\omega_2\wedge\ldots\wedge\omega_{p-1}$. Thus, taking the limit and applying the bounded convergence theorem, we obtain that 
    $$
    \int_V \tau \cdot [\omega_1]\cdot\ldots\cdot[\omega_{p-1}] \geq \int_V [\omega_1]^2\cdot[\omega_2]\cdot\ldots\cdot[\omega_{p-1}] > 0.
    $$
\end{proof}
Finally, we shall need the following statement, which will be used to establish statements related to wall-chamber decompositions.
\begin{lemma}\label{lem:wallchamberlemma}
   Let $X$ be a compact K\"ahler manifold and $\tau$ a $p+1$-modified K\"ahler class on $X$. Then, there exists an open neighbourhood $U$ of $\tau$ in $H^{1,1}(X,\mathbb R)$ and a finite set $S= \{V_1,\dots,V_r\}$ of irreducible $p$-dimensional subvarieties of $X$ such that for all K\"ahler forms $\omega_1,\dots,\omega_{p-1}$ and all $\tau^\prime\in U$ whenever we have \[
\int_V \tau \cdot [\omega_1] \cdot \dots [\omega_{p-1}] \leq 0
\]
then $V\in S$. 
\end{lemma}
\begin{proof}
    By the openness of the $(p+1)$-modified K\"ahler cone $\mathcal M_{p+1}\mathcal K$, we can find $\tau_1,\dots,\tau_\ell$ containing $\tau$ in the interior of their convex hull, 
\[
\tau \in U:= \mathrm{Int}(\mathrm{conv}(\tau_1,\dots,\tau_\ell)) \neq \varnothing.
\]
Since $\tau_i$ is $(p+1)$-modified K\"ahler, $E_{nK}(\tau_i)$ contains no $(p+1)$-dimensional irreducible subvarieties of $X$. Let $S$ be the union of all the $p$-dimensional irreducible subvarieties of $X$ contained in $E_{nK}(\tau_i)$ for $i = 1,\dots, \ell$. Now, let $\tau^\prime \in U$ and $\omega_1,\dots, \omega_{p-1}$ be K\"ahler classes on $X$. Suppose that we have 
\[
\int_V \tau^\prime \cdot [\omega_1] \cdot \dots\cdot [\omega_{p-1}] \leq 0
\]
for some $p$-dimensional irreducible subvariety $V$ of $X$. Then, since $\tau^\prime \in U$, we can write 
\[
\tau^\prime = \sum_{i=1}^\ell a_i \tau_i
\]
where $a_i > 0$ and $\sum_i a_i = 1$. From this, it follows that 
\[
\int_V \tau_i \cdot [\omega_1] \cdot \dots\cdot [\omega_{p-1}] \leq 0,
\]
for some $i$, and this in turn implies (by Lemma \ref{lem:main lemma}) that $V$ lies in $E_{nK}(\tau_i)$, and therefore $V\in S$. 
\end{proof}

\bigskip

\section{Destabilising subvarieties for the J-equation} \label{Section main results J}
\noindent In the spirit of slope stability, we will use the following notation: if $X$ is a compact Kähler manifold and $\alpha,\beta$ are Kähler classes on $X$, set 
$$
\mu_{\alpha,\beta}(Z) := (\dim Z) \frac{\alpha^{\dim Z - 1}\cdot \beta \cdot [Z]}{\alpha^{\dim Z}\cdot [Z]}.
$$
We will abbreviate $\mu_{\alpha,\beta}(X) =: \mu_{\alpha,\beta}$. The triple $(X,\alpha,\beta)$ is called $J$-stable (respectively $J$-semistable) if for all proper irreducible analytic subvarieties $Z$ of $X$, we have 
$$
\mu_{\alpha,\beta}(Z) < \mu_{\alpha,\beta} \quad \textrm{ (respectively } \mu_{\alpha,\beta}(Z) \leq \mu_{\alpha,\beta}\textrm{)},
$$
and $J$-unstable if it is not $J$-semistable. We will often drop the prefix and refer to these notions as simply stable or semistable respectively. 

\subsection{For threefolds}

We are now in a position to prove one of the main results of this paper.
\begin{theorem} \label{main thm 3-folds}
    Let $X$ be a compact Kähler 3-fold and let $\alpha, \beta$ be Kähler classes on $X$ such that the triple $(X,\alpha, \beta)$ is J-semistable and the class $\mu_{\alpha,\beta} \alpha - 2 \beta$ is big. Then, there exist at most finitely many irreducible subvarieties $Z\subseteq X$ such that 
    $$
    \mu_{\alpha,\beta}(Z) = 
    \mu_{\alpha, \beta}.
    $$
    \end{theorem}
    \begin{proof}
       Observe that $(X,\alpha,\beta)$ is semistable but not stable precisely if there exists an irreducible proper subvariety $Z$ of $X$ satisfies $\mu_{\alpha,\beta}(Z) = \mu_{\alpha,\beta}$, i.e. if
        $$
        (\mu_{\alpha,\beta} \alpha - (\dim Z) \beta)\cdot \alpha^{\dim Z - 1} \cdot [Z] = 0.
        $$ Because $ \mu_{\alpha,\beta}\alpha - 2\beta$ (and hence also $\mu_{\alpha,\beta}\alpha - \beta$) is big, by the lemma above, such a $Z$ necessarily satisfies $Z \subseteq E_{nK}(\mu_{\alpha,\beta}\alpha - (\dim Z) \beta)$. 
        Thus, if an irreducible surface $S$ satisfies $\mu_{\alpha,\beta}(S) = \mu_{\alpha,\beta}$, then $S \subseteq E_{nK}(\mu_{\alpha,\beta}\alpha - 2 \beta)$ implies necessarily that $S$ must be one of the finitely many irreducible surface components of $E_{nK}(\mu_{\alpha,\beta}\alpha - 2\beta)$. 

        Similarly, if $C$ is an irreducible curve satisfying $\mu_{\alpha,\beta}(C) = \mu_{\alpha,\beta}$, then $C \subseteq E_{nK}(\mu_{\alpha,\beta} \alpha - \beta)$. If $C$ is not one of the finitely many irreducible components of this analytic set, then $C \subseteq S$ for some irreducible surface in $E_{nK}(\mu_{\alpha,\beta} \alpha - \beta)$. Suppose first that $S$ is a smooth surface. Then, we have
        $$
         (\mu_{\alpha,\beta}(S)\alpha -\beta)^2\cdot [S] = \beta^2 \cdot [S] > 0, 
         $$
         and
         $$(\mu_{\alpha,\beta}(S)\alpha-\beta)\cdot\alpha \cdot [S]= \alpha\cdot\beta\cdot[S] > 0.
        $$
        From this, it follows that the class $(\mu_{\alpha,\beta}(S)\alpha-\beta)|_S$ is big on $S$ and therefore by the \cite[Proposition 3.1]{SohaibDyrefelt} there exist only finitely many curves $C$ on $S$ such that 
         $
        (\mu_{\alpha,\beta}(S)\alpha - \beta) \cdot [C] \leq 0.
        $
        Moreover, by semistability of $(X,\alpha,\beta)$ we have $\mu_{\alpha,\beta}(S) \leq \mu_{\alpha,\beta}$, so 
        $
        (\mu_{\alpha,\beta}\alpha - \beta)\cdot [C] \leq 0
        $
        for only finitely many curves $C \subset S$, as desired. 
        
        In case $S$ is singular, let $f: \tilde X \to X$ be any birational morphism such that the proper transform $\tilde S$ of $S$ in $\tilde X$ is smooth and an isomorphism away from the singular set of $S$. (Such a morphism always exists by Hironaka's theorem on resolution of singularities.) Then, $f_* [\tilde S] = [S]$ in homology and by the projection formula, we have 
        $$
        \mu_{f^*\alpha,f^*\beta}(\tilde S) = \frac{2 f^*(\alpha \cdot \beta) \cdot [\tilde S]}{f^* (\alpha^2) \cdot [\tilde S]} = \frac{2 \alpha \cdot \beta \cdot [S]}{\alpha^2\cdot [S]} = \mu_{\alpha,\beta}(S).
        $$
        Similarly, $\mu_{f^*\alpha,f^*\beta}=\mu_{f^*\alpha,f^*\beta}(\tilde X) = \mu_{\alpha,\beta}$. From this, we obtain (after another application of the projection formula) that  
        $$
         (\mu_{f^*\alpha,f^*\beta}(\tilde S) f^*\alpha - f^*\beta)^2 \cdot [\tilde S] = f^*(\mu_{\alpha,\beta}(S)\alpha -\beta)^2\cdot [\tilde S] = f^*\beta^2 \cdot [\tilde S] = \beta^2 \cdot [S]> 0.
        $$
        Moreover, since $\alpha$ is a Kähler class, $f^*\alpha$ can be represented by a semipositive form on $\tilde X$ (and hence also on $\tilde S$) and we have 
        $$
        (\mu_{f^*\alpha,f^*\beta}(\tilde S) f^*\alpha - f^*\beta)\cdot f^* \alpha \cdot [\tilde S] = f^*(\alpha \cdot \beta) \cdot [\tilde S] = \alpha \cdot \beta \cdot [S] > 0.
        $$
        This implies, by Lamari's criterion (see \cite[Theorem 5.5]{Lamari}) that the class $f^*(\mu_{\alpha,\beta}(S) \alpha - \beta)$ is a big class on $\tilde S$. 
        As before, it follows from semistability of $(X,\alpha,\beta)$ that
        $$
        (\mu_{\alpha,\beta}f^*\alpha - f^*\beta) - (\mu_{\alpha,\beta}(\tilde S)f^*\alpha - f^*\beta) = (\mu_{\alpha,\beta} - \mu_{\alpha,\beta}(S)) f^* \alpha
        $$
        is nef on $\tilde X$, therefore also on $\tilde S$, and in particular, the class $f^*(\mu_{\alpha,\beta}\alpha - \beta)\vert_{\tilde S}$ is a big class on $\tilde S$.
        Hence one concludes from a standard result (see, for example, \cite[Proposition 3.1]{SohaibDyrefelt}) that there are finitely many curves $\tilde C$ in $\tilde S$ (each of which is an irreducible component of the negative part of the Zariski decomposition of $f^*(\mu_{\alpha,\beta}\alpha - \beta)\vert_{\tilde S}$) such that 
        \[
        \int_{\tilde C} f^*(\mu_{\alpha,\beta}\alpha - \beta)\vert_{\tilde S} \leq 0.
        \]
         This last inequality is equivalent to 
        \[
        \mu_{\alpha,\beta} (f^*\alpha)\cdot[\tilde C] \leq (f^*\beta)\cdot [\tilde C].
        \]
        Once again, by the projection formula, this is equivalent to the statement that there are only finitely many curves $\tilde C$ in $\tilde S$ such that
        \[
        \mu_{\alpha,\beta} (\alpha\cdot f_*[\tilde C]) = \beta\cdot f_*[\tilde C].
        \]
        Now let $C \subseteq S$ be a curve such that $\mu_{\alpha,\beta}=\mu_{\alpha,\beta}(C)$ and suppose $C$ does not lie entirely in the singular locus of $S$. Then its proper transform $\tilde C$ under $f$ satisfies $f_*[\tilde C] = [C]$. Now observe that 
        \[
        \mu_{\alpha,\beta}=\mu_{\alpha,\beta}(C) \iff \mu_{\alpha,\beta} (\alpha\cdot[C])=\beta\cdot [C] \iff \mu_{\alpha,\beta} (\alpha\cdot f_*[\tilde C]) = \beta\cdot f_*[\tilde C]
        \]
        and therefore $\tilde C$ can only be one of the finitely many $\tilde C$ in $\tilde S$ with this property, and this implies that there are at most finitely many destabilising curves $C \subseteq S$. 

       We have therefore proven that the destabilising subvarieties are among the finite list of
        \begin{enumerate}
            \item irreducible curve components $C$ of $E_{nK}(\mu_{\alpha,\beta} \alpha - \beta)$
            \item irreducible surface components $S$ of $E_{nK}(\mu_{\alpha,\beta}\alpha - 2\beta)$
            \item the irreducible curve components of the singular locus of $S$ as $S$ ranges over the surface components of $E_{nK}(\mu_{\alpha,\beta}\alpha-\beta)$
            \item for every irreducible surface component $S$ of $E_{nK}(\mu_{\alpha,\beta}\alpha - \beta)$ the curves $C\subseteq S$ not lying entirely in the singular locus of $S$ whose strict transform, under any resolution of singularities $f: \tilde X \to X$ of $S$ as above, occurs as an irreducible component of the negative part of the Zariski decomposition of the (big) class $f^*(\mu_{\alpha,\beta}\alpha-\beta)|_{\tilde S}$ on $\tilde S$. 
        \end{enumerate}
        This concludes the proof.
    \end{proof}

\begin{remark}
    We note that the choice of resolution $f:\tilde X\to X$ of the irreducible surface $S \subseteq E_{nK}(\mu_{\alpha,\beta}\alpha-\beta)$ does not affect the set of potential destabilisers described in the proof. More precisely, let $f_i:\tilde X_i \to X$ be two resolutions of $S$ coming from Hironaka's proof, i.e. the strict transform $\tilde S_i$ of $S_i$ in $\tilde X_i$ is smooth and $f_i$ is a composition of blowing up repeatedly in points or smooth curves contained in the singular locus of (the strict transforms of) $S$.  Then, by resolving the indeterminacy of the rational map $f_2^{-1} \circ f_1$ we get a common resolution $g :\tilde X \to X$ with morphisms $g_i:\tilde X \to \tilde X_i$ with $f_i \circ g_i = g$ for $i=1,2$. If we denote by $\tilde S$ the strict transform of $S$ in $\tilde X$, then by what we said above about Hironaka's proof, $g_i:\tilde S \to S_i$ is a composition of blowups in points of the smooth surface $S_i$.
    Now the formulas (see \cite[Lemma 5.8]{Boucksomthesis})
    $$
    (g_i)_* N(g_i^*f_i^*(\mu_{\alpha,\beta}\alpha-\beta)) = N(f_i^*(\mu_{\alpha,\beta}\alpha-\beta))
    $$ imply that the support of the  negative part $N(g^*(\mu_{\alpha,\beta}\alpha-\beta))$ comprises the strict transforms of the irreducible components of $N(f_i^*(\mu_{\alpha,\beta}\alpha-\beta))$ plus a possible union of a divisor supported on the exceptional locus of $g_i$. Since this is true for $g_1$ and $g_2$, it follows that for each curve $C \subseteq S$, if the strict transform $\tilde C_1 \subseteq \tilde S_1$ under $f_1$ occurs as a component of $N(f_1^*(\mu_{\alpha,\beta}\alpha-\beta))$, then its strict transform $\tilde C_2 \subseteq \tilde S_2$ under $f_2$ also occurs as a component of $N(f_2^*(\mu_{\alpha,\beta}\alpha-\beta))$, and vice versa.
\end{remark}

\noindent Away from semistability, it may happen that there are infinitely many destabilising subvarieties, even under the assumption that the class $\mu_{\alpha,\beta}\alpha-2\beta$ be big. However, in this case we still have adequate control over the destabilising subvarieties in the following precise sense: 
\begin{theorem}\label{thm analyticity}
    Suppose $X$ is a compact Kähler 3-fold and $\alpha,\beta$ are Kähler classes on X such that $\mu_{\alpha,\beta}\alpha-2\beta$ is a big class. Let $V_{\alpha,\beta}$ be the union of all irreducible subvarieties $Z$ of $X$ with $\mu_{\alpha,\beta}(Z) \geq \mu_{\alpha,\beta}$. Then, $V_{\alpha,\beta}$ is a proper analytic subset of $X$.
\end{theorem}
\begin{remark}
    In particular, $V_{\alpha,\beta}$ has finitely many irreducible components. 
\end{remark}
\begin{proof}
    If $S$ is any destabilising irreducible surface in $X$, then $S \subseteq E_{nK}(\mu_{\alpha,\beta}\alpha - 2 \beta)$. So, there are always at most finitely many destabilising surfaces. Let $V_2$ be the union of all of these. 
    
    We must prove that all but finitely many destabilising curves $C$ are contained in $V_2$. But a destabilising curve $C$ likewise satisfies $C \subseteq E_{nK}(\mu_{\alpha,\beta}\alpha-\beta)$. There are only finitely many irreducible components of $E_{nK}(\mu_{\alpha,\beta}\alpha - \beta)$ of dimension one. Assume therefore that $C$ is not an irreducible component of $E_{nK}(\mu_{\alpha,\beta}\alpha - \beta)$. Then $C$ lies instead in a surface $S \subseteq E_{nK}(\mu_{\alpha,\beta}\alpha - \beta)$. If $S \subseteq V_2$, there is nothing to check. So, suppose $S$ is not contained in $V_2$, i.e. $\mu_{\alpha,\beta}(S) < \mu_{\alpha,\beta}$. Then, by the same argument as in the proof above, there exist at most finitely many curves $C \subseteq S$ with $\mu_{\alpha,\beta}(C) \geq \mu_{\alpha,\beta}(S)$ and therefore also only finitely many curves $C$ with $\mu_{\alpha,\beta}(C) \geq \mu_{\alpha,\beta} > \mu_{\alpha,\beta}(S)$. This proves the claim.
\end{proof}

\noindent The proof of the theorem in particular yields the following useful corollary: 

\begin{corollary}\label{cor:finiteatsemistability}
    Let $(X,\alpha,\beta)$ be as in the Theorem. Then, for each surface $S\subseteq V_{\alpha,\beta}$ satisfying $\mu_{\alpha,\beta}(S)\leq \mu_{\alpha,\beta}$ there are only finitely many curves $C \subseteq S$ such that $\mu_{\alpha,\beta}(C)\geq \mu_{\alpha,\beta}.$
\end{corollary}
\begin{proof}
    Observe that $\mu_{\alpha,\beta}(S) \leq \mu_{\alpha,\beta}$ is all one needs to conclude that the class $\mu_{\alpha,\beta}\alpha-\beta$ is big on $S$.
\end{proof}

\begin{lemma}
    Let $\alpha, \beta_i, i = 1, 2$ be Kähler classes on a compact Kähler 3-fold $X$ such that $\mu_{\alpha,\beta_i}\alpha - 2\beta_i$ is a big class for $i=1,2$. Then for any $t\in [0,1]$ we have 
    $$
    V_{\alpha,\beta_1} \cap V_{\alpha,\beta_2} \subseteq V_{\alpha,t\beta_1+(1-t)\beta_2} \subseteq V_{\alpha,\beta_1}\cup V_{\alpha,\beta_2}.
    $$
\end{lemma}
\begin{proof}
    This is an immediate consequence of 
    \begin{equation}\label{eqn:linearityofslope}
    \mu_{\alpha,t\beta_1 +(1-t) \beta_2}(Z) = t\mu_{\alpha,\beta_1}(Z)+(1-t)\mu_{\alpha,\beta_2}(Z).
    \end{equation}
    
\end{proof}

\begin{proposition}\label{prop:finitelymanyirreds}
    Let $\alpha, \beta_i$ be Kähler classes for $i=1,\dots,s$ on a compact Kähler 3-fold $X$ such that $\mu_{\alpha,\beta_i}\alpha - 2\beta_i$ is a big class for each $i$. Then as $\beta$ ranges over the convex hull of $\beta_1,\dots,\beta_s$, the collection of irreducible subvarieties of $X$ occurring as irreducible components $V_{\alpha,\beta}$ is a finite set.
\end{proposition}
\begin{proof}
    By the lemma, we have that 
    \begin{equation}\label{eqn:convexinclusion}
        V_{\alpha,\beta} \subseteq \bigcup_{i=1}^s V_{\alpha,\beta_i}
    \end{equation}
    as $\beta$ ranges $\mathrm{conv}(\beta_1,\dots,\beta_s)$. This implies that the irreducible surface components of $V_{\alpha,\beta}$ are among the finitely many irreducible surface components of $V_{\alpha,\beta_i}$ for $i = 1, \dots, s$. Therefore, we only need to prove that the irreducible curve components of $V_{\alpha,\beta}$ only take values in a finite set. 

    Denote by $S_1,\dots,S_n$  the finitely many irreducible surfaces occurring as irreducible components of all the $V_{\alpha,\beta_i}$ for $i = 1, \dots, s$, and let $C$ be an irreducible curve component of some $V_{\alpha,\beta}$ for $\beta \in \mathrm{conv}(\beta_1,\dots,\beta_s)$. Then, by (\ref{eqn:convexinclusion}) we have $C \subseteq V_{\alpha,\beta_i}$ for some $i$. If $C$ is not contained in any $S_i$, then $C$ can only be one of the finitely many irreducible curve components of $V_{\alpha,\beta_i}$. Thus, all the curves $C$ that occur as irreducible components of some $V_{\alpha,\beta}$ not contained in one of the $S_i$ are among the finitely many irreducible curve components of $V_{\alpha,\beta_i}$ for $i=1,\dots,s$. It remains to treat the case of curves lying in some $S_i$. 

    For each $i=1,\dots,n$, consider the following partition
    $$
    K:= \mathrm{conv}(\beta_1,\dots,\beta_s) = K^{(i)}_1 \cup K^{(i)}_2,$$
    
    where
    $$
    K^{(i)}_1:= \{ \beta \in K \ | \ \mu_{\alpha,\beta}(S_i) \leq \mu_{\alpha,\beta} \}, \quad K^{(i)}_2 := \{ \beta \in K \ | \ \mu_{\alpha,\beta}(S_i) > \mu_{\alpha,\beta} \}.
    $$
    Then, it is clear that $S_i$ occurs as an irreducible component of $V_{\alpha,\beta}$ for each $\beta \in K^{(i)}_2$. On the other hand, $\beta \mapsto \mu_{\alpha,\beta}(S_i)-\mu_{\alpha,\beta}$ is a linear map. Thus, the closed set $K^{(i)}_1$ is also the convex hull of finitely many points, say $\beta^{(i)}_1,\dots,\beta^{(i)}_{r_i}$. Now, if $\beta \in K^{(i)}_1$ and $C$ is an irreducible component of $V_{\alpha,\beta}$ lying in $S_i$, then, by definition, we have $\mu_{\alpha,\beta}(C) \geq \mu_{\alpha,\beta}$. But then (\ref{eqn:linearityofslope}) implies that $$
    \mu_{\alpha,\beta^{(i)}_j}(C) \geq \mu_{\alpha,\beta^{(i)}_j}
    $$
    for some $1\leq j \leq r_i$. But we also have that  $\mu_{\alpha,\beta^{(i)}_j}(S_i) \leq \mu_{\alpha,\beta^{(i)}_j}$ and thus, by the Corollary \ref{cor:finiteatsemistability}, there are only finitely many curves $C$ in $S_i$ such that $\mu_{\alpha,\beta^{(i)}_j}(C) \geq \mu_{\alpha,\beta^{(i)}_j}.$ Therefore, the irreducible components of $V_{\alpha,\beta}$ for $\beta \in K$ are among the following finite list:
    \begin{enumerate}
        \item the irreducible curve components of $E_{nK}(\mu_{\alpha,\beta_i}\alpha-\beta_i)$ for $i = 1, \dots, s$.
        \item the irreducible surface components $S_1,\dots, S_n$ of $E_{nK}(\mu_{\alpha,\beta_i}\alpha-2\beta_i)$ for $i = 1, \dots, s$.
        \item the finitely many curves $C$ in $S_i$ such that $$\mu_{\alpha,\beta^{(i)}_j}(C)\geq \mu_{\alpha,\beta^{(i)}_j}$$ for $i=1,\dots, n$, and $j = 1,\dots,r_i$.
    \end{enumerate}
\end{proof}

\subsection{Rigidity}
Aside from the production of a finite list of subvarieties that destabilise (or optimally destabilise) the J-equation, it is also desirable to note that they have certain special properties in common. Most notably, we have the following rigidity statement in the case of threefolds. 

\begin{proposition}\label{prop:rigidity}
    Let $(X,\alpha,\beta)$ and $V_{\alpha,\beta}$ be as in the statement of Theorem \ref{thm analyticity}, with $X$ projective. Then every irreducible component of $V_{\alpha,\beta}$ is rigid. More precisely, every irreducible surface component of $V_{\alpha,\beta}$ is the unique effective cycle representing its homology class, and every irreducible curve component $C$ of $V_{\alpha,\beta}$ satisfies the following: for every irreducible surface $S$ containing $C$, either $C$ is an irreducible component of the singular locus of $S$ or the strict transform of $C$ under any resolution of singularities of $S$ is a curve of negative self-intersection. 
\end{proposition}
\begin{proof}
    Let $S$ be a surface component of $V_{\alpha,\beta}$. Then the conclusion about $S$ follows immediately from \cite[Proposition 3.13]{Boucksomthesis} noting that $S$ is then an irreducible component of the negative part of the divisorial Boucksom-Zariski decomposition, and hence exceptional. 
    Now let $C$ be a curve component of $V_{\alpha,\beta}$. Then, by the definition of $V_{\alpha,\beta}$, it follows that if $S^\prime$ is any irreducible surface that contains $C$, then $\mu_{\alpha,\beta}(S^\prime) < \mu_{\alpha,\beta}$. Since $X$ is projective, we can always find some $S^\prime$ with this property. But then the proof of \ref{thm analyticity} shows that either $C$ is contained in the singular locus of $S$ or the proper transform of $C$ is an irreducible component in the negative part of the Zariski decomposition of the class $\mu_{\alpha,\beta}(S^\prime)\alpha-\beta$ on any resolution of singularities of $S^\prime$. Thus, $C$ has negative self-intersection as a divisor in any resolution of singularities of $S^\prime$.
\end{proof}

\subsection{Extension of Theorem \ref{main thm 3-folds} to dimension $\geq 4$}

In the theory of divisorial Zariski decompositions the convex \emph{modified K\"ahler} (and \emph{modified nef}) cones in $H^{1,1}(X,\mathbb{R})$ play central roles (see Section \ref{Section prelim}). The same argument as for 3-folds can still be carried out on compact K\"ahler manifolds of arbitrary dimension as long as the restrictions of certain classes are always big. This is guaranteed by assuming a stronger positivity conditions, as per the following statement:

\begin{theorem} \label{main theorem higher dim}
Let $X$ be a compact K\"ahler manifold and $\alpha,\beta$ K\"ahler classes on $X$. Suppose moreover that $\mu_{\alpha,\beta}\alpha - p\beta \in \mathcal M_{p+1}\mathcal K$ is $(p+1)$-modified K\"ahler on $X$ for every $p = 1,2,\dots,n-1$. Then, there exist at most finitely many irreducible subvarieties $V\subseteq X$ such that 
    \[
    \mu_{\alpha,\beta}(V) \geq \mu_{\alpha,\beta}(X) = \mu_{\alpha, \beta}.
    \]
\end{theorem}

\begin{proof}
Let $1\leq p \leq n-1$ and $Z$ an irreducible analytic subvariety of $X$ of dimension $p$. Note that $\mu_{\alpha,\beta}(Z)\geq \mu_{\alpha,\beta}$ is equivalent to 
\[
\int_V (\mu_{\alpha,\beta}\alpha - p \beta)\cdot \alpha^{p - 1} \leq 0,
\]
which implies, by Lemma \ref{lem:main lemma} that $Z \subseteq E_{nK}(\mu_{\alpha,\beta}\alpha-p\beta)$. But the class $\mu_{\alpha,\beta}\alpha-p\beta \in \mathcal M_{p + 1}\mathcal K$ is a $(p+1)$-modified Kähler class. Thus, $Z$ must be one of the finitely many $p$-dimensional irreducible components of the proper analytic subset $E_{nK}(\mu_{\alpha,\beta}\alpha - p\beta)$ because this set does not contain any irreducible subvarieties of dimension greater than or equal to $p+1$ thanks to Lemma \ref{lem:EnK}.
\end{proof}

\noindent An improved version for optimal destabilisers, proving Theorem \ref{main thm 1 intro} (4), is the following:
\begin{theorem}\label{Thm main higher dim optimal}
Let $X$ be a compact K\"ahler manifold and $\alpha,\beta$ K\"ahler classes on $X$ such that $(X,\alpha,\beta)$ is not J-stable, i.e. $\Delta^\mathrm{pp}_{\beta}(\alpha) \leq 0$. Suppose moreover that 
\begin{equation} \label{Eq 1111} 
\left(\mu_{\alpha,\beta} - (n-p)\Delta^\mathrm{pp}_{\beta}(\alpha)\right)\alpha - p\beta \in \mathcal{M}_{p+1}\mathcal{K}_X
\end{equation} 
for all $p = 2,\dots,n-1$ (note that we do not require this for $p = 1$ here). Then $\mathrm{Dest}_{\alpha,\beta}^{\mathrm{opt}}$ is a non-empty and finite. 

Moreover, there is a uniform constant $\lambda > 0$ depending only on $(X,\alpha,\beta)$ such that all except finitely many subvarieties $V \subset X$ satisfy
    $$
    \frac{1}{\dim X-\dim V}\left(\mu_{\alpha,\beta} - \mu_{\alpha,\beta}(V)\right) \geq  
    \Delta^{\mathrm{pp}}_{\beta}(\alpha) + \lambda.
    $$
\end{theorem}
\begin{remark}
Another way of thinking of this hypothesis is to consider $\mu_{\alpha,\beta'}\alpha - p\beta'$ for $\beta' := \beta 
- \Delta^{\mathrm{pp}}_{\beta}(\alpha)\alpha$, as used in the proof below. Note that the hypothesis \eqref{Eq 1111} is invariant under replacing $\beta$ by $\beta + r\alpha$ where $r \in \mathbb{R}$. In addition, since $(X,\alpha,\beta)$ is assumed unstable, by definition $\Delta^{\mathrm{pp}}_{\beta}(\alpha) < 0$, and hence the above theorem in particular holds under the stronger hypothesis that $\mu_{\alpha,\beta}\alpha - p\beta \in \mathcal{M}_{p+1}\mathcal{K}_X$.
\end{remark}
\begin{proof}[Proof of Theorem \ref{Thm main higher dim optimal}]
Write
$$
\beta' := \beta 
- \Delta^{\mathrm{pp}}_{\beta}(\alpha)\alpha
$$
and note that 
$$
\mu_{\alpha,\beta'}\alpha - p\beta'= 
\left(\mu_{\alpha,\beta} - (n-p)\Delta^{\mathrm{pp}}_{\beta}(\alpha)\right)\alpha - p\beta,
$$
as well as
$$
\Delta^{\mathrm{pp}}_{\beta'}(\alpha) = 0.
$$
Hence for this choice of $\beta'$, the triple $(X,\alpha,\beta')$ is semistable and also $\mu_{\alpha,\beta'}\alpha - p\beta'$ is $(p+1)-$modified K\"ahler for all $p = 2,3,\dots,n-1$, which implies that there is a finite number of optimally destabilising subvarieties of dimension $p \geq 2$. It remains to show that there are only finitely many optimal destabilising curves. By the hypothesis for $p = 2$ the class $\mu_{\alpha,\beta'}\alpha - 2\beta'$ is assumed to lie in $\mathcal{M}_{3}\mathcal{K}_X$, and hence $\mu_{\alpha,\beta'}\alpha - \beta'$ is in $\mathcal{M}_{3}\mathcal{K}_X$ as well. Thus the non-Kähler locus $E_{nK}(\mu_{\alpha,\beta'}\alpha - \beta')$ contains only irreducible components which are either curves or surfaces (see Lemma 2.3). Therefore any destabilising curve is either a) an irreducible curve component of $E_{nK}(\mu_{\alpha,\beta^\prime}\alpha - \beta')$, or b) contained in an irreducible surface component $S$ of $E_{nK}(\mu_{\alpha,\beta'}\alpha - \beta')$. In the former case there are finitely many such curves by proper analyticity of the non-K\"ahler locus (recall that any $3$-modified K\"ahler class is automatically big). In the latter case, if $S$ is smooth, restricting the K\"ahler classes $\alpha$ and $\beta$ to $S$ the main result of \cite{SohaibDyrefelt} shows that there are only finitely many destabilising curves for $(S,\alpha_{\vert S}, \beta'_{\vert S})$, i.e. such that $\mu_{\alpha,\beta'}(S) \leq \mu_{\alpha,\beta'}(C)$. The case of $S$ singular is treated as in the proof of Theorem \ref{main thm 3-folds}. Finally, we invoke semistability of $(X,\alpha,\beta')$, i.e. $\mu_{\alpha,\beta'}(V) \leq \mu_{\alpha,\beta'}(X)$ for every irreducible subvariety $V \subset X$. In particular $\mu(S) \leq \mu(X)$ and any destabilising curves $C \subset X$ satisfies $\mu_{\alpha,\beta'}(C) = \mu(X)$. Putting this together we have $$\mu_{\alpha,\beta'}(S) \leq \mu_{\alpha,\beta'}(X) = \mu_{\alpha,\beta'}(C).$$ 
Under the assumption that $(X,\alpha,\beta')$ is semistable, any destabilising curve for $(X,\alpha,\beta')$ must therefore also be a destabilising curve for $(S,\alpha_{\vert S}, \beta'_{\vert S})$ after restriction; and there are finitely many of those by \cite{SohaibDyrefelt}. In conclusion, the positivity hypothesis corresponding to $p = 1$ is not needed in the semistable case, thanks to Lamari's criterion on surfaces.
    
Combined, this implies that $\mathrm{Dest}^{\mathrm{opt}}_{\alpha,\beta'}$ is a finite set. It is moreover non-empty since the infimum defining $\Delta^\mathrm{pp}_{\beta'}(\alpha)$ is realised by a finite number of subvarieties (see also \cite{SD4}). Finally, since
$$
\mathrm{Dest}^{\mathrm{opt}}_{\alpha,\beta'} = \mathrm{Dest}^{\mathrm{opt}}_{\alpha,\beta},
$$
this concludes the proof of the first part.

For the second part, start from $(X,\alpha,\beta')$ which is semistable and such that $\mu_{\alpha,\beta'}\alpha - p\beta' \in \mathcal{M}_{p+1}\mathcal{K}_X$ for all $p = 2,3,\dots,n-1$. By openness of the modified K\"ahler cones, fix $\lambda > 0$ such that 
while $(X,\alpha,\beta' - \lambda\alpha)$ is no longer semistable, the classes $(\mu\alpha - p\beta' - (n-p)\lambda\alpha)_{\vert S_i}$ remain big on each of the finitely many irreducible surface components of $E_{nK}(\mu_{\alpha,\beta'}\alpha - \beta)$. We then conclude as in the proof of Theorem \ref{main thm 3-folds} that all except a finite number of subvarieties $V \subset X$, $\dim V = p$, satisfy
$$
\int_{V} (\mu_{\alpha,\beta'}\alpha - p\beta' - (n-p)\lambda\alpha) \wedge \alpha^{p-1} > 0.
$$
This implies that 
$$
\frac{\int_{V} (\mu_{\alpha,\beta'}\alpha - p\beta' - (n-p)\lambda\alpha) \wedge \alpha^{p-1}}{\int_V (n-p)\alpha^p} = \frac{\int_{V} (\mu_{\alpha,\beta}\alpha - p\beta) \wedge \alpha^{p-1}}{\int_V (n-p)\alpha^p} - \Delta^{\mathrm{pp}}_{\beta}(\alpha) - \lambda =
$$
$$
= \frac{1}{n-p} \left(\mu_{\alpha,\beta} - \mu_{\alpha,\beta}(V)\right) - \Delta^{\mathrm{pp}}_{\beta}(\alpha) - \lambda > 0.
$$
Rearranging terms yields the desired conclusion. 
\end{proof}

\begin{remark}
In Section \ref{Section example} we discuss the hypothesis of $p$-modified K\"ahler classes, and show that it holds in many natural situations.
\end{remark}

\section{Necessity of positivity conditions 
and first examples}\label{Section example}

\noindent In our previous work \cite{SohaibDyrefelt} on surfaces as well as that of \cite{DatarMeteSong} a key observation was that the $(1,1)$-class $\mu\alpha - 2\beta$ playing the central role in the proof, is \emph{always} big, based on Lamari's criterion \cite{Lamari}. We would expect that producing an analogue of this in higher dimension will be of importance to further work also on, for example, the conjectures \cite{DatarMeteSong} of Datar-Mete-Song. 

In particular, it is natural to ask when the hypotheses in the main theorem \ref{main thm 3-folds} are optimal. 

\subsection{Necessity of positivity assumptions on $\mu\alpha - p\beta$} \label{Subsection ex} As a first remark, we show that the hypotheses in \ref{main thm 1 intro} are necessary, in the sense that if we drop these hypotheses there are examples with an infinite number of destabilisers. In particular, the hypothesis that $\mu\alpha - (n-2)\beta$ be $(n-1)$-modified K\"ahler cannot be dropped in general, as illustrated already for blowups of $\mathbb{P}^n$ for $n\geq 3$:  
\begin{example}
\emph{(Infinitely many destabilisers if $\mu\alpha - (n-2)\beta$ is not $(n-1)$-modified K\"ahler)} 
Let $X = Bl_p \mathbb P^n$ be the blowup of $\mathbb P^n$ in a point $p \in \mathbb P^n$. Let $\alpha,\beta$ be Kähler classes with $\alpha = aH-bE = 3H-\frac{1}{9}E, \beta = cH-dE= \frac{1}{3}H-\frac{1}{9}E$ where $H$ is the pullback of the class of a linear hyperplane and $E$ is the exceptional divisor. One checks that
\[
\mu_{\alpha,\beta} = n \frac{\alpha^{n-1}\cdot\beta\cdot [X]}{\alpha^n\cdot [X]} = n\frac{a^{n-1}c-b^{n-1}d}{a^n-b^n} = n\frac{3^{n-2}-\frac{1}{9^n}}{3^n-\frac{1}{9^n}} = n\frac{3^{3n-2}-1}{3^{3n}-1} < \frac{n}{3}.
\]
On the other hand, the class of any linear hyperplane $P\cong \mathbb{P}^{n-2}$ contained in $E\cong \mathbb{P}^{n-1}$ is $-[E]\cdot[E]$ and using this, we can calculate 
\[
\mu_{\alpha,\beta}(P) = (n-2)\frac{\alpha^{n-3}\cdot \beta \cdot [P]}{\alpha^{n-2}\cdot [P]} = (n-2)\frac{d}{b}=n-2.
\]
If $n \geq 3$, then we see that $\mu_{\alpha,\beta} < \frac{n}{3} \leq n-2=\mu_{\alpha,\beta}(P)$ and so there are infinitely many destabilisers in this case. However, note that the inequality $\mu_{\alpha,\beta}\leq \mu_{\alpha,\beta}$ is equivalent to 
\[
\int_P (\mu_{\alpha,\beta}\alpha-(n-2)\beta)\cdot \alpha^{n-2}
\leq 0.
\]
Since $E\cong \mathbb{P}^{n-1}$ is a projective space, this shows that $(\mu_{\alpha,\beta}\alpha - (n-2)\beta)|_E$ is not Kähler, and hence not big. It follows that $\mu_{\alpha,\beta}\alpha-(n-2)\beta$ is not $(n-1)$-modified Kähler by \cite[Proposition 2.4]{Boucksomthesis}. (Recall that $(n-1)$-modified Kähler is simply called modified Kähler in \cite{Boucksomthesis}.)
\end{example} 

\noindent The above simple example motivates a deeper study of the imposed hypotheses that $\mu\alpha - p\beta$ be $(p+1)$-modified K\"ahler. Other aspects are investigated in the following subsection. 

\subsection{Pairwise distinct 
modified K\"ahler cones: The Cutkoski construction and examples of Wu} \label{Subsection ex Cutkosky} 
Here, we recall an example (a special case of an example found in \cite{Wu}) of a manifold $X$ of dimension $n$ such that on $X$ the cones $\mathcal M_{p}\mathcal K$ of $p$-modified Kähler classes are all pairwise distinct (except, of course, $\mathcal M_1\mathcal K = \mathcal M_0 \mathcal K = \mathcal K$). To begin with, we proceed in slightly more generality than we need.
Let $Y$ be a smooth projective variety of dimension $d_0$ and $L$ a very ample line bundle on $Y$. Let $E$ denote the vector bundle \[E:= \mathcal O_Y \oplus (L^{\otimes a_1})^{\oplus b_1} \oplus \dots \oplus (L^{\otimes a_r})^{\oplus b_r}\] on $Y$, where $a_0:= 0 < a_1 < a_2 < \dots < a_{r}$ is a strictly increasing sequence of positive integers and each $b_i$ is a positive integer. We shall denote by $X:= \mathbb P (E^\vee)$ the projective bundle of one-dimensional subspaces of the fibres of $E\to Y$ and $\pi:X\to Y$ the associated projection map. Note that $X$ has dimension $n = d_0+b_1+\dots+b_r$. Then (for $r \geq 1$) we have \[
H^{1,1}(X,\mathbb R) = \mathbb R H \oplus \pi^* H^{1,1}(Y,\mathbb R) 
\] where $H = \mathcal O_{\mathbb P(E)}(1)$ is the dual of the tautological sub-line bundle of $\pi^* E$. For $i,j = 1, \dots, r$, denote by $D_{ij}$ the divisor associated to the (dual of the) surjection \[
E \twoheadrightarrow \mathcal O_Y \oplus (L^{\otimes a_1})^{\oplus b_1} \oplus \dots \oplus \widehat{(L^{\otimes a_i})} \oplus \dots \oplus (L^{\otimes a_r})^{\oplus b_r}.
\] where $\widehat{(L^{\otimes a_i})}$ means we omit the $j$-th  summand equal to $L^{\otimes a_i}$. Note that the class of $D_{ij}$ is equal to $H-a_iL$ (where we denote by $L$ again the pullback $\pi^*L$). We can apply \cite[Proposition 4]{Wu} to deduce that the class $aL+bH$ is pseudoeffective (respectively nef) on $X$ if and only if $b\geq 0$ and $a+a_r b \geq 0$ (respectively $b \geq 0$ and $a\geq 0$). Note that if $aL+bH$ is pseudoeffective, then we can write \[
aL+bH = (a+a_r b)L+bD_{rj} \textrm{ for any } j = 1, \dots, b_r
\] whence we see (since $(a+a_r b)L$ is nef) that the non-nef locus of every pseudoeffective class of the form $aL+bH$ satisfies \[
E_{nn}(aL+bH) \subseteq \bigcap_{j=1}^{b_r} D_{rj}.
\] In fact, the non-nef loci of pseudoeffective classes are much more constrained on $X$. Indeed, whenever $a+a_pb \geq 0$ for any $p=0,1,\dots,r$, we can write the decomposition \[
aL+bH = (a+a_pb)L + bD_{pj} \textrm{ for any } j = 1, \dots, b_p.
\] This shows that $\nu(aL+bH,x) = 0$ for $x \not\in V_p$  where \[V_p := \bigcap_{i\geq p} \bigcap_{j=1}^{b_i} D_{ij}.\] In other words, if $a+a_pb\geq 0$ and $b\geq 0$, the non-nef locus of $aL+bH$ is contained in $V_p$, which is a $d_p$-dimensional smooth subvariety of $X$, with $d_p := d_0 + b_1 + \dots + b_p$. This proves that \[
\mathcal M_{d_p+1}\mathcal K \supseteq \{ aL+bH \ \vert \ a+a_pb > 0, b > 0 \}. 
\] On the other hand, let $aL+bH$ be a big class, i.e. $a+a_rb > 0$ and $b>0$. Then, according to the proof of \cite[Proposition 5]{Wu} we get that \[
\nu(aL+bH,V_{p+1})\geq \min \{ t\geq 0 \ \vert \ 0 \in (a + t[a_{p+1},a_{r}] + (b-t)[0,a_p]) \}. 
\] This implies that if $a+a_p b < 0$ then we have \[\nu(aL+bH,V_{p+1}) \geq -\frac{a+a_pb}{a_{r}-a_p}>0.\] Since $d_{p+1} = \dim V_{p+1} \geq d_p + 1$, this shows that $aL+bH \not\in \mathcal M_{d_p+1}\mathcal K$ if $a+a_pb<0$. Thus, we have shown that for $p=1,\dots, n$ we have \begin{equation*}
\mathrm{span}_\mathbb R (L,H)\cap \mathcal M_{p}\mathcal K = \{ aL+bH \ \vert \ a+a_{q-1}b > 0, b > 0\}
\end{equation*}
where $q$ is the unique integer satisfying $d_{q-1} < p \leq d_q$. It is instructive to examine special cases of this construction.
\begin{enumerate}
    \item Let $Y$ be a Riemann surface, $r=n-1$, $L$ any line bundle on $Y$ of degree $d>0$. Set $b_1=b_2=\dots=b_{n-1}=1$, then $X=\mathbb P( \mathcal O_Y \oplus (L^\vee)^{\otimes a_1} \oplus \dots \oplus (L^\vee)^{\otimes a_{n-1}})$ is an $n$-dimensional projective manifold.  We see that the cones $\mathcal M_p \mathcal K$ are given by \begin{equation}\label{eqn:ConeCalculation}
    \mathcal M_p\mathcal K_X = \{aL+bH \ \vert \ a+a_{p-1}b>0, b>0 \} 
    \end{equation} and hence are all pairwise distinct for all $p=1,\dots, n$. Intersection theory on $X$ is given by the formulas \begin{equation}\label{eqn:IntersectionTheory}H^n = d\sum_{i=1}^{n-1}a_i, \quad H^{n-1}\cdot L = d, \quad H^{n-i}\cdot L^i = 0 \textrm{ for }i\geq 2.\end{equation}
    \item In the above special case, let $Y= \mathbb P^1, L= \mathcal O_{\mathbb P^1}(1), r=2$ and $a_1 = 1, a_2 = 3$. Then, $X=\mathbb P(\mathcal O_{\mathbb P^1} \oplus \mathcal O_{\mathbb P^1}(-1)\oplus \mathcal O_{\mathbb P^1}(-3))$. Let $\alpha = L + H$ and $\beta = L + bH$ for $b>0$. Let us examine all the possible destabilising subvarieties for the $J$-equation on $(X,\alpha,\beta)$. Denote by $S$ and $C$ respectively the subvarieties
    \[
    S:= \mathbb P(\mathcal O_{\mathbb P^1} \oplus \mathcal O_{\mathbb P^1}(-1)) \subseteq X, \quad C:= \mathbb P(\mathcal O_{\mathbb P^1})\subseteq X.
    \]
    Making use of \eqref{eqn:ConeCalculation} and \eqref{eqn:IntersectionTheory}, one can show that the class $\mu_{\alpha,\beta}\alpha-2\beta$ is big if and only if $b>1/15$. In the notation of Theorem \ref{thm analyticity} we then have 
    \[
    V_{\alpha,\beta}= 
    \begin{cases}
        S \textrm{ for } 1/15 <b\leq 5/26\\
        C \textrm{ for } 5/26<b\leq 2/9\\
        \varnothing \textrm{ for } b > 2/9.
    \end{cases}
    \]
    In fact, we have  \[
    \operatorname{Dest}_{\alpha,\beta}= 
    \begin{cases}
        \{S,C\} \textrm{ for } 1/15 <b\leq 5/26\\
        \{C\}\textrm{ for } 5/26<b\leq 2/9\\
        \varnothing \textrm{ for } b > 2/9,
    \end{cases}
    \] 
    whereas for $b\leq 1/29$, $\operatorname{Dest}_{\alpha,\beta}$ contains infinitely many curves. 
    \item Let $Y$ be a compact Kähler manifold of dimension $n$, $L$ an ample line bundle on $Y$, $r=1$ and $b_1 = m+1$. Then $X= \mathbb P(\mathcal O_Y \oplus (L^\vee)^{\oplus (m+1)})$ has dimension $n+m+1$. Then, for any big class of the form $aL+bH$, the non-nef locus is empty if and only if $a+b\geq0, b > 0$, and equal to $P= \mathbb P(\mathcal O_Y)\subseteq X$ if and only if $a+b <0, b>0$. Thus, (writing ${\mathcal M_p}^\prime  \mathcal K :=  \mathcal M_p\mathcal K\cap \mathrm{span}(H,L)$) we have 
    \[
    \mathcal M_{0}^\prime\mathcal K = \mathcal M_{1}^\prime\mathcal K = \dots = \mathcal M_{n-1}^\prime\mathcal K = \mathcal M_{n}^\prime\mathcal K \subsetneqq \mathcal M_{n+1}^\prime\mathcal K = \dots = \mathcal M_{n+m+1}^\prime\mathcal K.
    \] In this situation, Datar-Mete-Song \cite{DatarMeteSong} have shown that if $\alpha = L + aH$, $\beta = L + bH$, for $a,b>0$, and  $\omega_0$, $\theta$ are fixed Kähler metrics in $\alpha,\beta$ respectively satisfying a \emph{Calabi-ansatz}, then, in the case that $(X,\alpha,\beta)$ is $J$-unstable, the $J$-flow with initial metric $\omega_0$ converges smoothly away from $P$. It therefore seems natural to ask whether, under suitable hypotheses, in the setting of Theorems \ref{thm analyticity} and \ref{main theorem higher dim} the $J$-flow should always converge smoothly away from the union of some collection of destabilising subvarieties. 
\end{enumerate}

\subsection{Sufficient conditions when the main results \ref{main thm 1 intro}, \ref{Cor finite test intro}, \ref{cor wall chamb dec intro}, \ref{main thm gMA intro}  apply } \label{Subsection optimal dest Jstab} 

\noindent  We finish this section by discussing a sufficient condition for when Theorem \ref{main thm 3-folds} and its higher dimensional analogue \ref{main theorem higher dim} apply. Moreover, note that no examples in this paper exclude the possibility that there may always exist a finite number of \emph{optimal} destabilisers. Motivated by this, we discuss some results also regarding optimal destabilising subvarieties.

To introduce our condition we associate to a pair $(\alpha,\beta)$ of K\"ahler classes on $X$, a third nef class $\eta := \mathrm{proj}_{\alpha}(\beta)$, when it exists, which we call the \emph{projection of $\beta$ through $\alpha$}, defined as the unique $(1,1)$-class on $X$ which is nef, not K\"ahler, and of the form $(1-t)\beta + t\alpha$ for some $t > 0$. Equivalently, $\eta$ is the second boundary point of the nef cone on the affine
line through \(\beta\) and \(\alpha\). Note that while this is not well-defined for all pairs of K\"ahler classes $(\alpha,\beta)$, it can always be defined up to rescaling $(\alpha, \lambda\beta)$ for some $\lambda > 0$ large enough as necessary. 
The set of (optimal) destabilising subvarieties on $(X,\alpha,\lambda\beta)$ is then exactly equal to the set of (optimal) destabilising subvarieties for $(X,\alpha,\beta)$, since 
$$
\Delta^{\mathrm{pp}}_{\lambda\beta}(\alpha) = \lambda \Delta^{\mathrm{pp}}_{\beta}(\alpha),
$$
and therefore rescaling, if necessary, is harmless for what we prove in this section.

In what follows, we always implicitly assume that $(\alpha,\beta)$ are such that the associated class $\eta$ is well-defined.

This allows us to formulate the following result, which provides a useful sufficient condition for when our main results for the J-equation apply (Theorem \ref{main thm 1 intro} and Corollaries \ref{Cor finite test intro} and \ref{cor wall chamb dec intro}): 

\begin{proposition} \label{Prop suff cond}
Suppose that $X$ is a compact K\"ahler and $\alpha_0,\beta$ are K\"ahler classes such that $\eta := \mathrm{proj}_{\alpha_0}(\beta)$ is $k$-modified K\"ahler for some $3 \leq k \leq n-1$. Then there exists a K\"ahler class $\alpha$ such that 
\begin{itemize}
    \item $\eta = \mathrm{proj}_{\alpha}(\beta)$
    \item $(X,\alpha,\beta)$ is unstable 
    \item The classes $$
\gamma(p,\alpha,\beta) := \left(\mu_{\alpha,\beta} - (n-p)\Delta^\mathrm{pp}_{\beta}(\alpha)\right)\alpha - p\beta
$$
are $k$-modified K\"ahler for each $p = 1,2,\dots,n-1$.
\end{itemize}
In particular, Theorem \ref{main thm 1 intro} and Corollaries \ref{Cor finite test intro} and \ref{cor wall chamb dec intro} all apply to  $(X,\alpha,\beta)$. Moreover, this holds true for all $\alpha'$ of the form $(1-r)\eta + r\beta$ with $r > 0$ small enough.
\end{proposition}

\begin{remark}
Note that it is easy to find many examples when $\mathrm{proj}_{\alpha}(\beta)$ is big \emph{(i.e $n$-modified K\"ahler where $n := \dim X$)} as long as we consider a K\"ahler manifold $X$ for which the inclusion of the K\"ahler cone in the big cone is strict, i.e. $\mathcal{K}_X \neq \mathcal{B}_X$. 
\end{remark}

\begin{proof}[Proof of Proposition \ref{Prop suff cond}]

Fix $\alpha, \beta \in \mathcal{K}_X$ and $\eta := \mathrm{proj}_{\alpha}(\beta) \in \partial \mathcal{K}_X$ the projection of $\beta$ through $\alpha$. By definition of $\mathrm{proj}_{\alpha}(\beta)$ we can then write $\alpha := \alpha_R$, where 
$$\alpha_R := (1-R)\eta + R\beta$$ for some $R \in (0,1]$.  

We then have the following useful double inequality: 
\begin{lemma}\label{Lemma double ineq}

Suppose that $\alpha_t = (1-t)\eta + t\beta$ with $t \in (0,1)$. If there exists a $t_0 \in (0,1)$ such that $\Delta^\mathrm{pp}_{\beta}(\alpha_{t_0}) \leq 0$, then we have the following double inequality 
\begin{equation} \label{eq double ineq}
\mu_{\beta,\alpha_t} - (n-1)t^{-1} \leq \Delta^\mathrm{pp}_{\beta}(\alpha_t) \leq \frac{1}{n-1} \left(\mu_{\beta,\alpha_t} - t^{-1} \right)
\end{equation}
for all $t \in (0,1]$.
\end{lemma}

\begin{proof}[Proof of Lemma] 
This is essentially proven in \cite{SD4}, but for completeness we here include an argument. We first prove the second inequality in \eqref{eq double ineq}. Start by observing that if $(X,\alpha_t,\beta)$ is stable (i.e. $\Delta^\mathrm{pp}_{\beta}(\alpha_t) > 0$) then \cite{CollinsGabor} and \cite{DonaldsonJobservation} implies that $\mu_{\beta,\alpha_t}\alpha_t - \beta > 0$, which by a straightforward modification of the proof of  \cite[Lemma 16]{SD4} is equivalent to
$
\mu_{\beta,\alpha_t} - t^{-1} > 0.
$
To prove the lemma we then exploit linearity of $\Delta^\mathrm{pp}_{\beta}(\alpha)$ in the $\beta$-argument (see \cite{SD4}). To this end, fix any $t \in (0,1)$ and consider the corresponding class $\alpha_t$. Let $\eta^*$ be the unique nef but not K\"ahler class of the form $(1-t)\eta + t\alpha$ for $t > 1$. Equivalently, \(\eta^*\) is the second boundary point of the nef cone on the
affine line through \(\eta\) and \(\alpha\); its uniqueness follows from convexity of
the nef cone and the fact that \(\alpha\) is Kähler. Consider a second path
$$
\beta_s := (1-s)\alpha_{t_s} + s\eta^*
$$
where $t_s \in (0,1]$ is defined so that
$
\alpha_{t_s} := (1-t_s)\eta + t_s\beta = \alpha_t
$
(so for some $s_0 \in (0,1)$ we have $\beta_{s_0}=\beta$, $t_{s_0} = t$, and $\alpha_{t_{s}} = \alpha_t$ for all $s \in [0,1)$). Now consider, for short, the functions
$$
L(s) := \Delta^\mathrm{pp}_{\beta_s}(\alpha_{t_s}), \; \; R(s) := \frac{1}{n-1}\left(\mu_{\beta_s,\alpha_{t_s}} - {t_s}^{-1}\right).
$$
As recalled above we conclude that whenever $\Delta^\mathrm{pp}_{\beta_s}(\alpha_{t_s}) \geq 0$ then also $
\mu_{\beta_s,\alpha_{t_s}} - {t_s}^{-1} \geq 0
$, i.e. $L(s) \geq 0$ implies $R(s) \geq 0$. Moreover, at $s = 0$ (corresponding to $t_0 = 1$) we have $\Delta^\mathrm{pp}_{\beta_0}(\alpha_{t_0}) = 1$ and 
$$
\frac{1}{n-1}\left(\mu_{\beta_0,\alpha_{t_0}} - t_0^{-1} \right) = \frac{1}{n-1} \cdot (n-1) = 1,
$$
i.e. $L(0) = R(0) = 1$. 
Finally we claim that both sides of the sought inequality are linear in the $\beta$-argument, as we move along $\beta_s := (1-s)\alpha_{t_s} + s\eta^*$ for $s \geq 0$ (this is clear for $R(s)$, and for $L(s)$ this is the main theorem of \cite{SD4}). To summarise, we have $L(0) = R(0) = 1$, and $L(s) \geq 0$ implies $R(s) \geq 0$. Moreover, by assumption, there exists an $s$ such that $L(s) \leq 0$. Putting it all together we find that $L(s) \leq R(s)$ for all $s \geq 0$. 

This gives the second inequality in \eqref{eq double ineq}.  The first inequality is proven similarly.
\end{proof}

\noindent As an application of the lemma one computes
$$
\gamma(p,\alpha_R,\beta) \geq 
\mu_{\alpha_R,\beta} - \left(\frac{n-p}{n-1}\mu_{\alpha_R,\beta} - \frac{n-p}{n-1} \frac{1}{R} \right)\alpha_R - p\beta = 
$$
$$
= \left[\frac{p-1}{n-1}\mu_{\alpha_R,\beta} + \frac{n-p}{n-1}\frac{1}{R} \right]\left[(1-R)\eta + R\beta \right] - p\beta
$$
where as usual $\theta \geq \theta'$ means that the cohomology class $\theta - \theta'$ is nef. Using the shorthand notations
$$
f(p,R) := \frac{p-1}{n-1}\mu_{\alpha_R,\beta} + \frac{n-p}{n-1}\frac{1}{R}, \; \; C_R := 
n\frac{\eta \cdot \alpha_R^{n-1}}{\alpha_R^{n}}
$$
we thus have
\begin{equation}\label{Eq17}
\gamma(p,\alpha_R,\beta) \geq (1-R)f(p,R)\eta - (p-Rf(p,R))\beta. 
\end{equation}
Now note that
$$
p - f(p,R)R = 
\frac{1}{n-1}(p-1)(n-\mu_{\alpha_R,\beta}R) 
$$
$$
= \frac{p-1}{n-1}n\left(\frac{((1-R)\eta + R\beta) \cdot \alpha_R^{n-1}}{\alpha_R^n} - \frac{R\beta \cdot \alpha_R^{n-1}}{\alpha_R^n}\right) = \frac{p-1}{n-1}(1-R) n\frac{\eta \cdot \alpha_R^{n-1}}{\alpha_R^n}
$$
$$
 = \frac{p-1}{n-1}(1-R)C_R
$$
and moreover $f(p,R) > 0$ always holds, because $\alpha_R$ and $\beta$ are K\"ahler classes. Therefore, if $R \in (0,1)$, the right hand side of \eqref{Eq17} lies in any given convex open cone $\mathcal{C} \subset H^{1,1}(X,\mathbb{R})$ if and only if
$$
\eta - \left[\frac{p-1}{n-1} \frac{n - \mu_{\alpha_R,\beta}R}{(1-R)f(p,R)} \right]\beta = \eta - \left[\frac{p-1}{n-1} \frac{C_R}{f(p,R)} \right]\beta \in \mathcal{C}.
$$
Applying this to any of the convex open cones $\mathcal{M}_{k}\mathcal{K}_X$ yields the desired conclusion. Note in particular that Lemma \ref{Lemma double ineq} implies that the stability threshold $\Delta^\mathrm{pp}_{\beta}(\alpha_R)$ tends to minus infinity as $R$ approaches zero, so for small enough $R > 0$ this quantity is always negative.
\end{proof}

\noindent As a first consequence, we now deduce the following statement, which has the advantage that the hypothesis is possible to verify in many examples.

\begin{corollary}
Suppose $(X,\alpha,\beta)$ satisfies that $\eta := \mathrm{proj}_{\alpha}(\beta)$ is $3$-modified K\"ahler.
Then the conclusion of Theorem \ref{Thm main higher dim optimal} holds. Namely, there is a finite number of subvarieties achieving the following infimum
$$
\inf_{V \subset X} \left(\frac{1}{\dim X - \dim V}\right)(\mu_{\alpha,\beta} - \mu_{\alpha,\beta}(V)) = \Delta^{\mathrm{pp}}_{\beta}(\alpha),
$$
and there is a uniform $\lambda := \lambda(\alpha,\beta) > 0$ such that only finitely many subvarieties $V' \subset X$, satisfy
$$
\frac{1}{\dim X - \dim V'}(\mu_{\alpha,\beta} - \mu_{\alpha,\beta}(V')) \leq \Delta_{\beta}^\mathrm{pp}(\alpha) + \lambda.
$$
\end{corollary}

\begin{proof}
Combine Proposition \ref{Prop suff cond} and Theorem \ref{Thm main higher dim optimal}. 
\end{proof}

\noindent It is also interesting to note the following. Recalling that we have fixed $\alpha$ and chosen $R \in (0,1)$ so that $\alpha_R = \alpha$ we immediately conclude from the above that
$$
\gamma(1,\alpha,\beta) \geq \frac{(1-R)}{R}\eta,
$$
i.e. this class is always nef. In fact, it is clear that there are optimally destabilising curves precisely if there is equality in Lemma \ref{Lemma double ineq}, and this is equivalent to $\gamma(1)$ being nef but not K\"ahler. 
From this one may also extrapolate bounds for the other classes $\gamma(p,\alpha,\beta)$ for $p \geq 2$, yielding
$$
\gamma(p,\alpha,\beta) = \gamma(1,\alpha,\beta) - (p-1)(\beta - \Delta^\mathrm{pp}_{\beta}(\alpha)\alpha) \geq \frac{(1-R)}{R}\eta - (p-1)(\beta - \Delta^\mathrm{pp}_{\beta}(\alpha)\alpha),
$$
which is true for any given $\alpha,\beta \in \mathcal{K}_X$. From this can equally well conclude that if $\eta \in \mathcal{M}_k\mathcal{K}_X$ and $R > 0$ is small enough, then so is $\gamma(p,\alpha_R,\beta)$ for all $p = 1,2,\dots,n-1$.

The main motivation of the above sufficient criteria, is to ensure that the finiteness of (optimal) destabilising subvarieties holds for a large class of $n$-folds, possibly all compact K\"ahler manifolds. In this direction, the above Corollary can  be summarised in the following slightly more general way: 

\begin{theorem}
Let $X$ be a compact K\"ahler $n$-fold on which there is a strict inclusion of cones $\mathcal K_X \subsetneq  \mathcal{M}_3 \mathcal K_X$. Then, there exist $\alpha, \beta \in \mathcal K_X$ K\"ahler classes for which $(X,\alpha,\beta)$ is not stable, and Theorems \ref{main thm 1 intro}, \ref{main thm gMA intro} apply. In particular, there is a finite number of (optimal) destabilising subvarieties, i.e. $\mathrm{Dest}_{\alpha,\beta}$ and $\mathrm{Dest}_{\alpha,\beta}^{opt}$ are both finite sets. 
\end{theorem}

\noindent This holds in particular when $X$ is an $n$-fold of the form $X = \mathbb P (\mathcal O_Y \oplus (L^\vee)^{a_1}\oplus\dots\oplus (L^\vee)^{a_n})$ for some $a_1 < a_2 < \dots < a_n$ (see Section \ref{Subsection ex Cutkosky}), providing general families of manifolds in arbitrary dimension where our main results hold. 

It is moreover interesting to observe that the result holds on blow-ups, starting from any initial threefold, in the following sense: 

\begin{corollary}
Let $X$ be any compact K\"ahler threefold, and let $\pi: Y \rightarrow X$ be the blow up of $X$ at a point $x \in X$. If $\alpha,\beta$ are K\"ahler classes on $X$, then consider the nef and big classes $\pi^*\alpha, \pi^*\beta$ on $Y$ and write $\alpha_{\epsilon} := \pi^*\alpha - \epsilon[E]$, and $\beta_{\epsilon'} := \pi^*\beta - \epsilon'[E]$. If moreover  $0 < \epsilon <<  \epsilon'$ are both small enough, then the set
$
\mathrm{Dest}_{\alpha_{\epsilon},\beta_{\epsilon'}}^{\mathrm{opt}}
$
of optimally destabilising subvarieties is finite, and its elements are precisely the subvarieties with strictly positive generic Lelong numbers 
$$
\nu((\mu - \Delta^\mathrm{pp}_{\beta_{\epsilon'}}(\alpha_{\epsilon}))\alpha_{\epsilon} - (n-1)\beta_{\epsilon'},V) > 0
$$
\end{corollary}

\begin{proof}
Write $\epsilon' = \lambda \epsilon$ where $\lambda > 1$. Assume moreover that $\lambda > 1$ is large enough so that $\lambda \alpha - \beta$ is nef. Then $(1-t)\beta_{\lambda\epsilon} + t\alpha_{\epsilon}$ is on the boundary of the nef cone $\overline{\mathcal{K}_X}$, precisely if $(1-t)\lambda + t = 0$, that is $t = \frac{1}{\lambda-1}$. Hence 
$$
\eta := \mathrm{proj}_{\alpha_{\epsilon}}(\beta_{\lambda\epsilon}) = \pi^*(\lambda\alpha - \beta)\cdot \frac{1}{\lambda-1}
$$
and this class is nef and big, so that $\eta^n > 0$. Moreover
$$
(1-R)\frac{1}{\lambda-1}\pi^*(\lambda\alpha - \beta) + R(\pi^*\beta - \lambda\epsilon E) = \pi^*\alpha - \epsilon E
$$ 
if and only if $R = \frac{1}{\lambda}$. If $\epsilon > 0$ is small enough (and $\lambda > 0$ large enough, with the constraint that $\lambda\epsilon$ is bounded above by the relevant Seshadri constant, so that $\beta_{\lambda\epsilon}$ is K\"ahler) then by Proposition \ref{Prop suff cond} it follows that $\gamma(p,\alpha_{\epsilon}, \beta_{\lambda\epsilon})$ is a big cohomology class for all $p = 1,2,\dots,n-1$, for all $\lambda >> 1$ large enough (and $\epsilon > 0$ correspondingly small enough). 
The last assertion regarding Lelong numbers follows from Lemma \ref{lem:main lemma}. 
\end{proof}

\begin{remark}
We expect that the strategy of proof generalises beyond the J-equation, to partly cover the PDE considered in Section \ref{Section other PDE} below. Since this would require the introduction of heavy notational machinery, we leave these aspects for future work.
\end{remark}

\medskip

\section{Destabilising subvarieties for other PDE} \label{Section other PDE}

\noindent In this section we finally clarify how the above arguments in fact apply also to a broad class of geometric PDE, which include many special cases of the $Z$-critical equations introduced in \cite{DMS}, inverse $\sigma_k$ equations in the sense of \cite{DatarPingali} (including the J-equation), as well as the deformed Hermitian Yang-Mills equation under certain phase hypotheses.

\subsection{Generalised Monge-Amp\`ere equations} This class of PDE has been studied in the projective case by Datar-Pingali \cite{DatarPingali} under the name of \emph{generalised Monge-Amp\`ere (gMA) equations}. The same kinds of techniques may also be applied to more general $Z$-critical equations (introduced in \cite{DMS}) in the special case of rank 1. This special case has recently been studied in \cite{FangMa} by Fang-Ma, where their solvability is linked to a Nakai-Moishezon type criterion, generalising previous works of G. Chen \cite{GaoChen}, Song \cite{Song}, Datar-Pingali \cite{DatarPingali} and others. However, for the sake of clarity, we shall here illustrate the ideas in the case of gMA equations. 

In formulating the equation we let $(X,\theta)$ be a compact K\"ahler manifold and let $\alpha$ be a Kähler class on $X$. Moreover, let $c_1,\ldots, c_{n-1}$ be real constants and $f:X\to\mathbb R$ be a smooth function such that the cohomological condition 
$$
\int_X \frac{\alpha^n}{n!} = \sum_{k=1}^{n-1}\frac{c_k}{(n-k)!}\int_X [\theta]^k \cdot \alpha^{n-k}  + \int_X f \theta^n
$$
is satisfied. Then, the \emph{generalised Monge-Amp\`ere equation} seeks a smooth Kähler form $\omega \in \alpha$ such that
\begin{equation}\label{eqn:gMAcoh}
 \frac{\omega^n}{n!} = \sum_{k=1}^{n-1}\frac{c_k}{(n-k)!} \theta^k \wedge \omega^{n-k}  +  f \theta^n
\end{equation}

It will be convenient to employ the following natural shorthand notation:  $$\exp(\omega) := \sum_{k = 0}^n \omega^k/k! \quad\quad \mathbf\Theta:= \sum_{k=0}^{n-1} c_k \theta^k  + f\theta^n.$$
We shall also write $\eta^{[k,k]}$ for the degree $(k,k)$-part of a multi-degree differential form $\eta$. If we denote $\omega_{\varphi} := \omega_0 + \frac{i}{2\pi}\partial \bar{\partial}\varphi \in [\omega_0]= \alpha$, then in this notation, the generalised Monge-Amp\`ere equation seeks a K\"ahler potential $\varphi$ such that
\begin{equation} \label{eqn:gMA}
\exp(\omega_{\varphi})^{[n,n]}= (\mathbf{\Theta} \wedge \exp(\omega_{\varphi}))^{[n,n]}.
\end{equation}
Note especially that if $\mathbf\Theta = \mu_{\alpha,[\theta]}^{-1}\theta$ then \eqref{gMA eq} is just the J-equation. We refer to \cite{FangMa, DatarPingali} for details on this broad formalism, including how special cases of interest (e.g. J-equation, inverse Hessian equations, dHYM equation) can be interpreted within this framework.

The key result of \cite{DatarPingali} in the projective case and of \cite{FangMa} in general is that if we moreover assume certain positivity conditions on $\mathbf\Theta$, then the solvability of \eqref{gMA eq} can be characterised by a Nakai-Moishezon type criterion involving intersection numbers on $X$. More precisely, assume that $(X,\alpha, \mathbf\Theta)$ satisfy the positivity conditions in \cite[(1.2)]{DatarPingali}. Then, we have the following theorem. 

\begin{theorem}[Datar-Pingali, Fang-Ma]\label{thm: Nakai-Moishezon for gMA}
    Let $(X,\alpha,\mathbf\Theta)$ satisfy the positivity hypotheses above. Then, the following are equivalent.
    \begin{enumerate}
        \item There exists a smooth Kähler form $\omega \in \alpha$ that solves the gMA equation (\ref{eqn:gMA}).
        \item For every proper irreducible analytic subvariety $V\subseteq X$ we have
        \[
        \int_{V} \left(\exp(\alpha) \cdot (1 - [\mathbf{\Theta}])\right) > 0.
        \]
    \end{enumerate}
\end{theorem}

In the light of the above theorem, and following \cite{DatarMeteSong} we introduce the following natural definition.

\begin{definition} \label{Definition gMA semistable}
 Suppose $(X,\alpha,\mathbf\Theta)$ satisfy the positivity condition above. Then we say that $(X,\alpha,\mathbf\Theta)$ is \emph{(gMA-)semistable (resp. stable)} if for every proper irreducible analytic subvariety $V \subseteq X$ we have 
 \[
 \int_{V} \left(\exp(\alpha) \cdot (1 - [\mathbf{\Theta}])\right) \geq 0 \quad (\textrm{resp. } > 0).
 \]
 If $(X,\alpha,\mathbf{\Theta})$ is not semistable, we say that it is unstable. If $V$ violates the above inequality, we say that it  \emph{(gMA)-destabilises} the triple $(X,\alpha,\mathbf\Theta)$.
\end{definition}

\noindent In order to apply our techniques, we finally need to restrict the class of generalised Monge-Amp\`ere equations considered. We shall focus on those gMA equations whose associated numerical criterion given by Theorem \ref{thm: Nakai-Moishezon for gMA} `factorises'. We explain this more precisely now. Let us define the polynomials  
$$
P(y) := \sum_{k=1}^{n-1} c_k y^k
$$
and for $p=1,\ldots, \dim X - 1$, 
\begin{equation}\label{eqn:defnofQp}
Q_p(x,y):= (\exp(x)(1-P(y)))^{[p]}
\end{equation}
where by $R(x,y)^{[p]}$ we mean the degree $p$ homogeneous part of the power series $R(x,y)$. We note that the polynomials $Q_p(x,y)$ only depend on the coefficients $c_1,\dots, c_{n-1}$ and not on $\theta$ or $\alpha$.
\begin{definition}[Factorisable gMA equation]\label{def:factorisablegMAeq}
We say that the triple $(X,\alpha,\mathbf\Theta)$ defines a \emph{factorisable} gMA equation if for each $p= 1, \ldots,\dim X - 1$, the polynomial $Q_p(x,y)$ can be factorised as 
$$
Q_p(x,y) = (x-r_py)\tilde Q_p(x,y)
$$
where 
\begin{enumerate}
    \item the constant $r_p \geq 0$,
    \item the polynomial $\tilde Q_p(x,y)$ is a polynomial with nonnegative real coefficients.
\end{enumerate}
The cohomology class $\tau_{p}(\alpha,[\theta]):= \alpha - r_p\beta$ is called the \emph{associated factor class at dimension $p$}.
\end{definition}

\begin{example}[Inverse Hessian equations]
For $k\geq 1$, let $\mathbf\Theta = \kappa \theta^k$  where $\kappa$ is the uniquely determined cohomological constant 
$$
\kappa = \frac{n!\int_X [\theta]^k\cdot \alpha^{n-k}}{(n-k)!} >0. 
$$
This choice of $\mathbf\Theta$ defines an \emph{inverse Hessian equation}. Then, for $p<k$ we have 
$$
Q_p(x,y)= \frac{1}{p!}x^p
$$
and for $p\geq k$ we have 
$$
Q_p(x,y) = \frac{1}{p!}x^p- \frac{\kappa}{(p-k)!}x^{p-k}y^k = \left(x-\kappa_p^{1/k}y\right)\left(\frac{1}{p!}\sum_{r=0}^{p-1}\kappa_p^{r/k}x^{p-1-r}y^r\right),
$$
where $\kappa_p := p!\kappa/(p-k)! >0.$ Thus we see that all inverse Hessian equations are factorisable, with the associated factor class at each dimension $p=k,\ldots,\dim X -1$ given by \[\tau_{p}(\alpha,[\theta]):= \alpha - \kappa_p^{1/k}[\theta].\] When $k = 1$ this corresponds to the J-equation.
\end{example}
\begin{example}[gMA equation on three-folds]\label{example gMA on three-folds}
Let $\dim X = 3$ and define 
$$
\mathbf\Theta = c \theta + d\theta^2 + f\theta^3
$$   
where $c,d$ are real constants and $f$ is a smooth function such that the triple $(X,\alpha,\mathbf\Theta)$ satisfies the above positivity hypotheses. Then, we have 
$$
Q_2(x,y) = \frac{1}{2}x^2-cxy -dy^2 = \frac{1}{2}(x+(-c+\sqrt{c^2+2d})y)((x+(-c-\sqrt{c^2+2d})y)
$$
and 
$$
Q_1(x,y) = x - cy.
$$
Thus, we see that the triple $(X,\alpha,\mathbf\Theta)$ defines a factorisable gMA equation whenever $d>0$, with associated factor class at dimension two equal to \[\tau_2(\alpha,[\theta]) = \alpha + (-c-\sqrt{c^2+2d})[\theta].\]
\end{example}

In fact, the notion of a factorisable gMA is not any more restrictive than equations of the form \eqref{eqn:gMA}. More precisely, we have the following result
\begin{lemma}\label{lem:gMAfactorisable}
    Suppose $(X,\alpha,\mathbf\Theta)$ satisfies the cohomological condition \eqref{eqn:gMAcoh} and the positivity hypothesis in \cite[(1.2)]{DatarPingali}. Then, $(X,\alpha,\mathbf\Theta)$ defines a factorisable gMA equation.
\end{lemma}
\begin{proof}
    We must prove that if $Q_p(x,y)$ is defined by \eqref{eqn:defnofQp}, then it can be factorised as 
    \[
    Q_p(x,y) = (x-r_py)\tilde Q_p(x,y)
    \]
    where $r_p\geq 0$ and all the coefficients of $\tilde Q_p(x,y)$ are non-negative. It suffices to prove that if $h_p(x) = Q_p(x,1)$, then each $h_p(x)$ can be factorised as 
    \[
    h_p(x)=(x-r_p)g_p(x)
    \]
    with $r_p\geq 0$ and $g_p(x)$ is a polynomial with non-negative coefficients, for then the claim will follow by homogenising $h_p(x)$ and $g_p(x)$. 

    Now note that 
    \[
    h_p(x) = \frac{x^p}{p!} - \sum_{k=1}^{p} \frac{c_k}{(p-k)!} x^{p-k}.
    \]
    It is clear that $h_p(x) \to +\infty$ as $x\to +\infty$, and $h_p(\varepsilon)\leq 0$ for $\varepsilon > 0$ small enough, with strict inequality if some $c_k > 0$ for $k = 1,\dots p$. Therefore, $h_p(x)$ admits a non-negative real root. Let $r_p$ be the largest non-negative real root of $h_p(x)$. Then, $r_p = 0$ if and only if $c_k = 0$ for all $k=1,\dots,p$. Now, we may write
    \[
    h_p(x) = (x-r_p)g_p(x)
    \]
    where $g_p(x)$ is a polynomial with real coefficients. We must prove that all the coefficients of $g_p(x)$ are non-negative. Clearly, if $r_p = 0$, then all the $c_k = 0$ for $k=1,\dots,p$, so in that case $g_p(x) = x^{p-1}/p!$  certainly has non-negative coefficients. So we may assume that $r_p > 0$. Write $g_p(x) = P_p(x) - N_p(x)$, where both $P_p(x)$ and $N_p(x)$ have non-negative coefficients. If $N_p(x)\neq 0$, then let $k_0$ be the smallest integer such that the coefficient of $x^{k_0}$ in $N_p(x)$ is strictly positive. Since $g_p(x)$ has degree $p-1$, $k_0 < p$. Then, we have the identity
    \[
    \frac{x^p}{p!} - \sum_{k=0}^{p-1} \frac{c_{p-k}}{k!}x^k = xP_p(x)+r_pN_p(x) - (xN_p(x) + r_p P_p(x)).
    \]
    The monomial $x^{k_0}$ occurs with non-zero coefficient on the right-hand-side of this identity only in the terms  $r_p N_p(x)$ (where it has a strictly positive coefficient) and (possibly) in $xP_p(x)$ (where it has a non-negative coefficient). On the other hand, all the coefficients of $x^k$ with $k < p$ have non-positive coefficients on the left-hand-side of this identity. This is a contradiction. Hence, $N_p(x) = 0$ and $g_p(x) = P_p(x)$ has non-negative coefficients.  
\end{proof}

\noindent One key feature of the factor classes $\tau_p(\alpha,[\theta])$ is that they are descending in $p$. More precisely, we have the following lemma.

\begin{lemma} \label{lem:factorclassesdescend} Let the triple $(X,\alpha,\mathbf\Theta)$ define a factorisable gMA equation. Then, for each $p=1,\dots,n-2$, we have $\tau_p(\alpha,[\theta])\geq \tau_{p+1}(\alpha,[\theta])$ with strict inequality whenever $\tau_{p+1}(\alpha,[\theta]) \neq \alpha$.
\end{lemma}
\begin{proof}
    Since $(X,\alpha,\mathbf\Theta)$ defines a factorisable gMA equation, the polynomials $Q_p$ given by \eqref{eqn:defnofQp} factorise as \[
    Q_p(x,y)= (x-r_py)\tilde Q_p(x,y)
    \] 
    where $r_p\geq 0$ and $Q_p(x,y)$ has nonnegative coefficients. Note that this implies that $x\mapsto Q_p(x,1)$ has at most one positive real root (counted with multiplicity), namely $r_p$, and $Q_p(x,1)\to+\infty$ as $x\to+\infty$. But it is easy to verify by straightforward computation that
    \[
    \frac{\partial}{\partial x}Q_{p+1}(x,y) = Q_{p}(x,y).
    \]
    If $r_{p+1} = 0$, i.e. $\tau_{p+1}(\alpha,[\theta]) = \alpha$, then the this implies that $\tau_p(\alpha,[\theta]) = \alpha$ also, and so $\tau_{p}(\alpha,[\theta]) \geq \tau_{p+1}(\alpha,[\theta])$. So, suppose $r_{p+1} >0$. Then, the factorisation of $Q_{p+1}(x,y)$ above implies that the polynomial function $x \mapsto Q_{p+1}(x,1)$ has exactly one zero (counted with multiplicity) on the positive real axis, namely $r_{p+1}$, and therefore its slope at $r_{p+1}$ must be positive. But since \[Q_p(x,1)=\frac{d}{dx}Q_{p+1}(x,1)\] we see immediately $Q_p(r_{p+1},1)>0$. Now, $x\mapsto Q_p(x,1)$ in turn also has at most one positive root and therefore cannot change sign more than once on the positive real axis. Thus, $Q_p(x,1) > 0$ whenever $x\geq r_{p+1}$, whence $r_p < r_{p+1}$. This shows that $\tau_p(\alpha,[\theta])>\tau_{p+1}(\alpha,[\theta])$.
\end{proof}
\begin{theorem}
    Suppose $(X,\alpha,\mathbf\Theta)$ satisfy the positivity hypothesis above and define a factorisable gMA equation whose associated factor class $\tau_{p}(\alpha,[\theta]) \in \mathcal M_{p+1} \mathcal K$ at dimension $p$ belongs to the $(p+1)$-modified Kähler cone for $p=1,\dots,n-1$. Then, there exist at most finitely many subvarieties that (gMA-)destabilise $(X,\alpha,\mathbf\Theta)$.
\end{theorem}
\begin{proof}
    Any destabiliser $V$ of dimension $p$ satisfies
    \[
    \int_V \tau_{p}(\alpha,[\theta])\cdot \tilde Q_p(\alpha,[\theta]) \leq 0.
    \]
    By Lemma \ref{lem:main lemma}, $V$ must therefore be contained in $E_{nK}(\tau_{p}(\alpha,[\theta]))$. This is a proper analytic subset not containing any $(p+1)$-dimensional subvarieties, so $V$ must be one of its finitely many irreducible components. 
\end{proof}

\noindent We should also point out that we get a much better result when $\dim X = 3$, just as in the special case of the $J$-equation.
\begin{theorem}\label{thm:gMAdim3}
    Suppose $X$ is a compact Kähler manifold of dimension $\dim X = 3$, $\alpha,\beta\in\mathcal K_X$ Kähler classes, with $\omega\in\alpha,\theta\in\beta$ Kähler forms. Let $\mathbf\Theta = c_1 \theta + c_2 \theta^2 + f\theta^3 $ be such that $(X,\alpha,\mathbf\Theta)$ satisfies the cohomological constraint \eqref{eqn:gMAcoh} and the positivity condition given in \cite[(1.2)]{DatarPingali}. Suppose that the factor class $\tau_2(\alpha,\beta)$ at dimension $2$ associated to $(X,\alpha,\mathbf\Theta)$ is big, and $(X,\alpha,\mathbf\Theta)$ is gMA-semistable. Then, there are finitely many curves and surfaces that (gMA)-destabilise the triple $(X,\alpha,\mathbf\Theta)$.
    \end{theorem}
    \begin{proof}
        The argument is very similar to the proof of Theorem \ref{main thm 3-folds}. We sketch the parts that are identical and elaborate on the parts that are different. 
        Let us briefly recall the notation. The polynomials $Q_p(x,y)$ for $p = 1, 2$ are given by
        \[
        Q_2(x,y) = \frac{1}{2}x^2 - c_1 xy - c_2 y^2, \quad Q_1(x,y) = x - c_1y. 
        \]
        Observe that if $c_1=c_2 = 0$, then the equation reduces to the complex Monge-Amp\`ere equation and there are no destabilisers in this case. So we may assume that $c_1>0$ or $c_2 > 0$. When the triple $(X,\alpha,\mathbf\Theta)$ is (gMA)-semistable, a destabilising surface $S$ (respectively, a destabilising curve $C$) satisfies 
        \[
        \int_S Q_2 (\alpha,\beta) = 0, \quad \textrm{ (respectively, } \int_C Q_1(\alpha,\beta) = 0\textrm{).} 
        \]
        By what was said in Example \ref{example gMA on three-folds}, we can write 
        \[
        Q_2(\alpha,\beta) =\frac{1}{2} (\alpha-r_2\beta)\cdot (\alpha + s\beta)
        \]
        where 
        \[
        r_2 = c_1+\sqrt{c_1^2+2c_2}, \quad s = -c_1 + \sqrt{c_1^2+2c_2}.
        \]
        It is clear that $r_2 > 0,s \geq 0$. Recall also that $\tau_2(\alpha,\beta) = \alpha-r_2\beta$. Now, if $S$ is a (gMA)-destabilising surface, then by Lemma \ref{lem:main lemma} we must have $S\subseteq E_{nK}(\tau_2(\alpha,\beta))$, and there are only finitely many surfaces $S$ which are contained in $E_{nK}(\tau_2(\alpha,\beta))$. 
        Suppose $C$ is a (gMA)-destabilising curve. Then, we have 
        \[
        \int_C \alpha - c_1 \beta = 0.
        \]
        Now $\alpha-c_1\beta \geq \alpha - r_p\beta$ (either by direct observation, or by Lemma \ref{lem:factorclassesdescend}), so $\alpha-c_1\beta$ is a big class as well. Once again, by Lemma \ref{lem:main lemma} this implies that $C\subseteq E_{nK}(\alpha-c_1\beta)$. Thus, $C$ is either an irreducible curve component of $E_{nK}(\alpha-c_1\beta)$ or $C$ is contained in an irreducible surface component $S$ of $E_{nK}(\alpha-c_1\beta)$.  It therefore remains to show that for each of the finitely many surface components $S$ of $E_{nK}(\alpha-c_1\beta)$, there are only finitely many curves $C \subseteq S$ such that 
        \[
        \int_C \alpha-c_1\beta = 0.
        \]
        But now we note that 
        \[
        \frac{1}{2}\int_S (\alpha-c_1\beta)^2 = \int_S \left(\frac{1}{2}\alpha^2 - c_1 \alpha\cdot\beta + c_1^2\beta^2\right) = \int_S Q_2(\alpha,\beta) + \frac{1}{2}(c_1^2+2c_2)\int_S\beta^2 > 0,
        \]
        and 
        \[
        \frac{1}{2}\int_S (\alpha-c_1\beta)\cdot (\alpha + s\beta) = \int_S Q_2(\alpha,\beta) + \sqrt{c_1^2 + 2c_2}\int_S \beta\cdot(\alpha + s\beta) > 0.  
        \]
        This shows that, if $S$ is smooth, then $(\alpha-c_1\beta)|_S$ is a big class on $S$. Now we argue exactly as in the proof of \ref{main thm 3-folds} to conclude.
    \end{proof}
    \begin{remark}
        In fact, with exactly the same strategy of proof as in Section 3, one can also obtain the more general analogues of Theorem \ref{thm analyticity} and Proposition \ref{prop:rigidity} for the generalised Monge-Amp\`ere equations when $\dim X = 3$, making necessary adjustments just as in the proof of Theorem \ref{thm:gMAdim3}.
    \end{remark}
    \begin{theorem}\label{thm:gMAanalyticity}
    Suppose $X$ is a compact Kähler manifold of dimension $\dim X = 3$, $\alpha,\beta\in\mathcal K_X$ Kähler classes, with $\omega\in\alpha,\theta\in\beta$ Kähler forms. Let $\mathbf\Theta = c_1 \theta + c_2 \theta^2 + f\theta^3 $ be such that $(X,\alpha,\mathbf\Theta)$ satisfies the cohomological constraint \eqref{eqn:gMAcoh} and the positivity condition given in \cite[(1.2)]{DatarPingali}. Suppose that the factor class $\tau_2(\alpha,\beta)$ at dimension $2$ associated to $(X,\alpha,\mathbf\Theta)$ is big. Let $V_{\alpha,\mathbf\Theta}$ be the union of all irreducible subvarieties $Z$ of $X$ such that 
    \[\int_Z \exp(\alpha)\cdot(1-[\mathbf\Theta]) \leq 0.\] Then,
    \begin{enumerate}
        \item $V_{\alpha,\mathbf\Theta}$ is a proper analytic subset of $X$, 
        \item Every irreducible component of $V_{\alpha,\mathbf\Theta}$ is rigid in the sense of Proposition \ref{prop:rigidity}. 
    \end{enumerate}
\end{theorem}

\subsection{Supercritical deformed Hermitian Yang-Mills equations} \label{subsection supercritical dHYM} Let $X$ be a smooth projective variety and let $\alpha$ be a $(1,1)$-cohomology class on $X$ (not necessarily Kähler) and $\theta$ a Kähler metric on $X$, as above. Then the \emph{deformed Hermitian Yang-Mills (dHYM)} equation seeks a smooth form $\omega \in \alpha$ such that the nowhere-zero complex valued function $F_\theta(\omega):X\to \mathbb C^*$ defined by 
$$
(\theta + \i \omega)^n = F_\theta(\omega) \theta^n
$$
has constant argument. In order to state the \emph{supercritical phase hypothesis} we shall need a further assumption, namely that there exists a representative $\omega \in \alpha$ such that $F_\theta(\omega)$ takes values in some half-plane. If such a representative $\omega \in \alpha$ exists, then in particular
$$
Z_{[\theta]}(\alpha):= \int_X (\theta + \i \omega)^n \neq 0.
$$ Let $\varphi_{[\theta]}(\alpha) : 
= \arg Z_{[\theta]}(\alpha)$, and define $\Phi_\theta(\omega)$ by the formula
$$
\Phi_\theta(\omega) = \sum_{k=1}^n \arctan \lambda_i
$$
where $\lambda_i$ are the (real) eigenvalues of the endomorphism $\theta^{-1}\omega$. It can be shown that $\Phi_\theta(\omega)$ is a smooth real-valued function $X \to (-n\frac{\pi}{2},n\frac{\pi}{2})$. Moreover, under the assumption that $F_\theta(\omega)$ takes values in some half-plane, we also have 
$$
\Phi_\theta(\omega) = \arg F_\theta(\omega) + 2\pi k
$$
for some integer $k$. We define the \emph{lifted angle} $\hat\varphi_{[\theta]}(\alpha)$  by 
$$
\hat\varphi_{[\theta]}(\alpha):= \varphi_{[\theta]}(\alpha) + 2\pi k.
$$
It follows (see \cite[Lemma 2.4]{CollinsXieYau}) that under this assumption, the integer $k$ is determined uniquely by $[\theta]$ and $\alpha$, i.e. if $\omega, \omega^\prime \in \alpha$ and $\theta,\theta^\prime \in [\theta]$ are such that both $F_\theta(\omega), F_\theta^\prime(\omega^\prime)$ take values in some (potentially different) half-planes, then $\Phi_\theta(\omega) - \arg F_\theta(\omega) = \Phi_{\theta^\prime}(\omega^\prime) - \arg F_{\theta^\prime}(\omega^\prime)$. We say that $(X,[\theta],\alpha)$ satisfies the \emph{supercritical phase hypothesis} if the lifted angle $\hat\varphi_{[\theta]}(\alpha) \in ((n-2)\frac{\pi}{2},n\frac{\pi}{2})$, or equivalently the \emph{complementary lifted angle} \[\hat\phi_{[\theta]}(\alpha):= n\frac{\pi}{2} - \hat\varphi_\theta(\alpha) \in (0,\pi).\] In this notation, the deformed Hermitian Yang-Mills equation seeks a smooth $(1,1)$-form representing $\alpha$ such that 
\begin{equation}\label{eqn:dHYM}
    \operatorname{Im}\left(e^{-\hat\phi_{[\theta]}(\alpha)}(\omega+\i\theta)^n\right)= 0.
\end{equation}Under these assumptions, Chu-Lee-Takahashi \cite{ChuLeeTakahashi} prove the following numerical criterion characterising solvability of the dHYM equation. 

\begin{theorem}[Chu-Lee-Takahashi]
    Let $X$ be a smooth projective variety, $\beta$ a Kähler class on $X$ and $\alpha$ a $(1,1)$-cohomology class. Suppose $(X,\beta, \alpha)$ satisfies the supercritical phase hypothesis. Then, the following are equivalent. 
    \begin{enumerate}
        \item For every Kähler form $\theta \in \beta$, there exists a smooth $(1,1)$-form $\omega$ representing $\alpha$ solving the deformed Hermitian Yang-Mills equation (\ref{eqn:dHYM}).
        \item For all proper irreducible subvarieties $V \subseteq X$, we have 
        \[
        \int_V \left(\operatorname{Re}(\alpha +\i \beta)^{\dim V} - \cot \hat\phi_{\beta}(\alpha) \operatorname{Im}(\alpha + \i\beta)^{\dim V}\right) > 0.
        \]
    \end{enumerate}
\end{theorem}
\begin{definition} \label{Definition dHYM semistable}
 Suppose $(X,\beta,\alpha)$ satisfies the supercritical phase hypothesis. Then we say that $(X,\beta,\alpha)$ is \emph{(dHYM-)semistable (resp. stable)} if for every closed, proper, irreducible subvariety $V \subseteq X$ we have 
 \[
\int_{V} \left(\operatorname{Re}(\alpha +\i \beta)^{\dim V} - \cot \hat\phi_{\beta}(\alpha) \operatorname{Im}(\alpha + \i\beta)^{\dim V}\right) \geq 0 \; \; (\textrm{resp. } > 0).
 \]
 If $(X,\beta,\alpha)$ is not semistable, we say that it is unstable. If $V$ violates the above inequality, we say that it  \emph{(dHYM-)destabilises} the triple $(X,\beta,\alpha)$.
\end{definition}
From the point of view of the above theorem, we are naturally led to consider the following family of polynomials. Fix $\hat\phi \in (0,\pi)$ and $p\geq 0$, and define 
\[Q_{p,\hat\phi}^{\mathrm{dHYM}}(x,y):= \operatorname{Re}(x+iy)^p - (\cot \hat\phi) \operatorname{Im}(x+iy)^p.\]
\begin{lemma}
    The polynomials $Q_{p,\hat\phi}^{\mathrm{dHYM}}(x,y)$ admit the factorisation
    \begin{equation} \label{eqn:dHYM factorisation}
    Q_{p,\hat\phi}^{\mathrm{dHYM}}(x,y) = \prod_{l=0}^{p-1}\left(x-\cot\left(\frac{\hat\phi+l\pi}{p}\right)y\right).
    \end{equation}
\end{lemma}
\begin{proof}
    A straightforward calculation shows that 
    $Q^{\mathrm{dHYM}}_{p,\hat\phi}(x,y)$  
    admits $p$ distinct roots given by 
    \[
    (x_l,y_l) = \left(\cos\left(\frac{\hat\phi + l\pi}{p}\right),\sin\left(\frac{\hat\phi + l\pi}{p}\right)\right)
    \]
    for $l = 0, \ldots, p-1$. Now $\hat\phi \in (0,\pi)$ implies that $y_l > 0$ and the result follows.
\end{proof}

The above factorisation allows us to deduce the following analogue of Theorem \ref{main thm 3-folds}.

\begin{theorem}\label{thm:dHYM}
Let $X$ be a smooth projective variety of dimension $n \leq 4$, $\alpha,\beta$ Kähler classes on $X$. Suppose $(X,\beta, \alpha)$ satisfies the supercritical phase hypothesis. If $\dim X = 4$ suppose moreover that $\hat\phi_\beta(\alpha)\in (\pi/2,\pi)$. If the classes 
\[
\tau^{\mathrm{dHYM}}_p(\alpha, \beta) := \alpha - \cot\left(\frac{\hat\phi_{\beta}(\alpha)}{p}\right)\beta
\]
for $p=1,\ldots,\dim X - 1$ lie in the $(p+1)$-modified Kähler cone, then there exist at most finitely many (dHYM)-destabilising proper irreducible subvarieties $V \subseteq X$, i.e. proper subvarieties $V$ satisfying
\[
\int_V \left(\operatorname{Re}(\alpha +\i \beta)^{\dim V} - \cot \hat\phi_{\beta}(\alpha) \operatorname{Im}(\alpha + \i\beta)^{\dim V}\right) \leq 0.
\]
\end{theorem}
\begin{proof}
By hypothesis, the class $\tau^{\mathrm{dHYM}}_p(\alpha, \beta)$ is $(p+1)$-modified Kähler, and the classes 
\[
\alpha - \cot\left(\frac{\hat\phi_{\beta}(\alpha)+\pi l}{p}\right)\beta
\]
for $l = 1, \ldots, p-1$ are also Kähler. Indeed, $\hat\phi_\beta(\alpha)\in(0,\pi)$ implies that $\cot((\hat\phi_\beta(\alpha)+\pi)/2) < 0$, and $\hat\phi_\beta(\alpha)\in(\pi/2,\pi)$ implies that \[
\cot\left(\frac{\hat\phi_\beta(\alpha)+\pi}{3}\right) < 0, \quad \cot\left(\frac{\hat\phi_\beta(\alpha)+2\pi}{3}\right) <0.
\]

Thus, in view of the factorisation (\ref{eqn:dHYM factorisation}) and Lemma \ref{lem:main lemma}, any $p$-dimensional irreducible subvariety $V$ such that 
\[
\int_V \left(\operatorname{Re}(\alpha +\i \beta)^{p} - \cot \hat\phi_{\beta}(\alpha) \operatorname{Im}(\alpha + \i \beta)^{p}\right) \leq 0.
\]
must satisfy
\[
V\subseteq E_{nK}(\tau^{\mathrm{dHYM}}_p(\alpha, \beta)),
\]
but the latter set contains at most finitely many irreducible subvarieties of dimension $p$, since it does not contain any $(p+1)$-dimensional subvariety and is a proper analytic subset of $X$. Thus, $V$ must be one of the finitely many irreducible components of $E_{nK}(\tau_p^{\mathrm{dHYM}}(\alpha,\beta))$.
\end{proof}

Similar to the case of the $J$-equation, we obtain a slightly sharper result in the semistable case in the three dimensional setting.

\begin{theorem}
    Let $X$ be a smooth projective variety of dimension three. Let $\alpha,\beta$ be Kähler classes on $X$ such that $(X,\beta,\alpha)$ satisfies the supercritical phase hypothesis and is \emph{(dHYM-)}semistable. Suppose the class $\tau_2^{\mathrm{dHYM}}(\alpha,\beta)$ is big. Then, there are only finitely many \emph{(dHYM-)}destabilising subvarieties of $X$. 
    
\end{theorem}
\begin{proof}
    If $S$ is any destabilising surface in $X$, then the proof of the above theorem shows that $S\subseteq E_{nK}(\tau_2^{\mathrm{dHYM}}(\alpha,\beta))$ and so there are always at most finitely many destabilising surfaces.

    A destabilising curve $C$ must satisfy 
    \[
    C \subseteq E_{nK}(\tau_1^{\mathrm{dHYM}}(\alpha,\beta)) = E_{nK}(\alpha - \cot \hat\phi_{\theta}(\alpha) \beta).
    \]
    If $\hat\phi:= \hat\phi_\beta(\alpha) \in[\pi/2,\pi)$, then $\tau_1^{\mathrm{dHYM}}(\alpha,\beta)$ is a Kähler class, so there are no such curves. So, suppose $\hat\phi \in (0,\pi/2)$, i.e. $\cot \hat\phi > 0$. Now observe that by semistability, for any surface $S$, we have 
    \[
    0\leq \int_S (\alpha-\cot(\hat\phi/2)\beta)\cdot(\alpha-\cot((\hat\phi+\pi)/2)\beta) = \int_S(\alpha - \cot(\hat\phi)\beta)^2 - (\csc(\hat\phi))^2\int_S \beta^2,
    \]
    where we have used the identity $\cot x = \cot 2x + \csc 2x$. Thus, for any surface $S$ and any curve $C\subseteq S$, we have that
    \[
    \int_S \tau_1^\mathrm{dHYM}(\alpha,\beta)^2=\int_S (\alpha-\cot(\hat\phi)\beta)^2 \geq (\csc \hat\phi)^2\int_S\beta^2 > 0, \quad \int_C \tau_1^\mathrm{dHYM}(\alpha,\beta)\geq 0. 
    \]
    The second inequality follows by the semistability hypothesis. Thus, if $S$ is any surface component of $E_{nK}(\tau_1^{\mathrm{dHYM}}(\alpha,\beta))$, and $S$ is smooth, the class $\tau_1^{\mathrm{dHYM}}(\alpha,\beta)$ restricts to a nef class on $S$ which is big and therefore we conclude as in the proof of Theorem \ref{main thm 3-folds} above that $S$ contains only finitely many destabilising curves (for they are among the irreducible components of the negative part of the Zariski decomposition of the big class $\tau_1^\mathrm{dHYM}(\alpha,\beta)|_S$). If $S$ is singular, we again conclude as in the proof of Theorem \ref{main thm 3-folds} above by taking a resolution of $S$. 
\end{proof}
\begin{corollary}
    Let $X$ be a smooth projective variety of dimension three. Let $\alpha,\beta$ be Kähler classes on $X$ such that $(X,\beta,\alpha)$ satisfies the supercritical phase hypothesis. Suppose the class $\tau_2^{\mathrm{dHYM}}(\alpha,\beta)$ is big. Then, the union $V^{\mathrm{dHYM}}_{\alpha,\beta}$ of all subvarieties of $X$ that \emph{(dHYM-)}destabilise the triple $(X,\beta,\alpha)$ is a proper analytic subset of $X$. 
\end{corollary}
\begin{proof}
    This argument is very nearly identical to the proof of Theorem \ref{thm analyticity} above after making appropriate straightforward changes.
\end{proof}

\section{Wall-chamber decompositions for PDE} \label{Section chamber}
One key motivation behind results like Theorems \ref{thm:gMAanalyticity} and  \ref{thm analyticity} is to understand the nature of stability (and hence solvability) when we vary the PDEs in continuous families. This motivation comes from analogous questions in algebraic geometry, where one finds that if certain numerical invariants are fixed, the space of all stability data admits a locally finite wall-chamber decomposition, with the notion of stability of any given object being the same for each element in any given chamber. By analogy, one should therefore expect something similar to happen in the case of PDEs. In \cite{SohaibDyrefelt}, the authors do indeed orove that this holds in the case of surfaces for the $J$-equation, dHYM equation, and, under mild and natural hypotheses, for $Z$-critical equations. In this Section, we aim to show that in higher dimensions, under the positivity conditions described in the previous Sections, we are still able to get closely analogous results. These results are the first in the literature of their kind in higher dimensions.
\subsection{Wall-chamber decompositions for the J-equation}
 
Fix $$S= \{\alpha_1,\dots,\alpha_s\} \subseteq \mathcal{K}_X,$$ a finite set of K\"ahler classes on $X$. Given $\beta\in\mathcal K_X$ denote 
\[
\mathfrak{M}_\beta = \{\alpha \in S \ | \ (X,\alpha,\beta) \textrm{ is stable}\}.
\]
An interesting conjecture, inspired by variation of stability in algebraic geometry, is the following.  
\begin{conj} \label{Conj wc dec}
There exists a locally finite collection of closed real submanifolds $W_\sigma$ of codimension one of the Kähler cone $\mathcal K_X$ (that is, for each $\beta \in \mathcal K_X$, there exists an open neighbourhood of $\beta$ which meets at most finitely many of the $W_\sigma$) such that for each connected component $K$ of 
\[
\mathcal K_X \setminus \bigcup_\sigma W_\sigma
\]
and each $\beta,\beta^\prime \in K$, we have that $\mathfrak{M}_\beta = \mathfrak{M}_{\beta^\prime}$.
\end{conj}
\noindent Fixing the set $S$ of Kähler classes $\alpha$ and varying the auxiliary class $\beta$ corresponds roughly to varying the choice of stability condition, and this can be considered a rather simple analogue of the phenomena for Bridgeland stability for the special case of the $J$-equation. Here, $\mathfrak M_\beta$ serves as a very simple prototype of the corresponding moduli space.
As an application of our main results Theorem \ref{main thm 3-folds} (or \ref{main theorem higher dim}), we can prove this conjecture for a subset of all $\beta \in \mathcal{K}_X$ :
\begin{theorem}
Conjecture \ref{Conj wc dec} is true with $\mathcal K_X$ replaced by $\mathcal K^\prime$, the open convex 
cone comprising the classes $\beta \in \mathcal K_X$ such that for each $\alpha \in S$ and $p = 1 , 2, \ldots,\dim X-1$, the class $\mu_{\alpha,\beta}\alpha - p\beta$ lies in the $(p+1)$-modified Kähler cone.
\end{theorem}
\begin{proof}
    For every $\alpha \in S$ and every irreducible subvariety $Z\subseteq X$, let $W_{(Z,\alpha)}$ be the locus given by $\beta\in \mathcal K^\prime$ such that $\mu_{\alpha,\beta}(Z) = \mu_{\alpha,\beta}$. Let $K$ be any connected component of \[
    \mathcal K^\prime \setminus \bigcup_\sigma W_\sigma
    \]
    as $\sigma$ ranges over all pairs $(Z,\alpha)$ with $Z \subseteq X$ subvarieties and $\alpha\in S$. If $\beta_1,\beta_2 \in K$ then $\mu_{\alpha,\beta_i}-\mu_{\alpha,\beta_i}(Z)$ for $i=1,2$ have the same sign for each $(Z,\alpha)$, so the triples $(X,\alpha,\beta_i)$ are either both $J$-semistable or not, and so $\alpha \in \mathfrak{M}_{\beta_1}$ if and only if $\alpha \in \mathfrak{M}_{\beta_2}$. 
    
    It only remains to prove that for any $\beta\in \mathcal K^\prime$, there exists an open neighbourhood $U\subseteq \mathcal K^\prime$ which meets at most finitely many of the loci $W_\sigma$. 
    Since $\mathcal K^\prime$ is an open convex cone, we can pick $\beta_1,\ldots,\beta_s \in \mathcal K_X^\prime$ such that $\beta$ lies in the interior $U$ of the convex hull $\operatorname{conv}(\beta_1,\ldots,\beta_s)$. By Theorem \ref{main theorem higher dim}, for each of the finitely many $\alpha \in S$, there exists a (possibly empty) finite set $\mathcal D_{i,\alpha}$ of irreducible subvarieties $Z$ satisfying $\mu_{\alpha,\beta_i}(Z)\geq \mu_{\alpha,\beta_i}$. If $\tilde \beta \in U$ lies on some $W_\sigma$, then $\mu_{\alpha,\tilde\beta}(Z) = \mu_{\alpha,\tilde\beta}$ for some $\alpha\in S$ and some $Z$. By linearity of $\tilde\beta\mapsto \mu_{\alpha,\tilde\beta}(Z)$, it follows that $\mu_{\alpha,\beta_i}(Z) \geq \mu_{\alpha,\beta_i}$ for some $1\leq i \leq s$. But this implies that $Z \in \mathcal D_{i,\alpha}$. Thus, the only walls meeting $U$ are $W_{(Z,\alpha)}$ as $\alpha$ ranges over the finite set $S$ and $Z$ ranges over the finite set $\bigcup_i \mathcal D_{i,\alpha}$.
\end{proof}

\subsection{Wall-chamber decompositions for gMA equations}
As a final application of these ideas, we illustrate how our arguments for the $J$-equation can be generalised to a wider class of PDEs. We choose the setting of generalised Monge-Amp\`ere equations, all of which satisfy the condition of factorisability introduced in Definition \ref{def:factorisablegMAeq}.

Let $X$ be a smooth projective variety, with $\alpha,\beta\in\mathcal K_X$ Kähler classes and $\theta\in\beta$ a fixed Kähler form. Recall that a gMA equation is given by data $(X,\alpha,\mathbf\Theta)$ where 
\[
\mathbf\Theta = \sum_{k=1}^{n-1} c_k \theta^k + f\theta^n
\]
with the $c_k\geq 0$ and $f$ satisfying the cohomological condition \eqref{eqn:gMAcoh} and the  the required positivity condition given in \cite[(1.2)]{DatarPingali}. Recall moreover that the numerical criterion for gMA equations says that the equation \eqref{eqn:gMA} is solvable precisely when 
\[
\int_V \exp(\alpha)\cdot (1-[\Theta]) > 0
\]
for all proper irreducible subvarieties of $X$.

For the sake of clarity, we make another simplifying assumption. We assume that $f=c_n$ is a constant. In this case, the positivity condition \cite[(1.2)]{DatarPingali} simplifies considerably, and is equivalent to demanding that the $c_k$ are not all zero.  This includes the case of all inverse Hessian equations. 

The space of gMA equations with $f=c_n$ a constant is in a natural way a codimension one closed submanifold of $\mathbb R_+[y]_{n}\times \mathcal K_X\times \mathcal K_X$ where $\mathbb R_+[y]_{n}$ denotes the set of degree $n$ non-zero real polynomials in one variable $y$ with zero constant term and all of whose coefficients are either zero or strictly positive. (We note that $\mathbb R_+[y]_n\times\mathcal K_X \times \mathcal K_X$ is in a natural way a manifold with corners.) We now explain this correspondence. Let $(P(y),\alpha,\beta)\in \mathbb R_+[y]_n \times \mathcal K_X\times \mathcal K_X$ satisfy the cohomological condition \eqref{eqn:gMAcoh} given by
\[
\int_X \exp(\alpha)\cdot(1-P(\beta)) = 0.
\]
This locus of triples $(P(y),\alpha,\beta)$ is a smooth submanifold because it is the zero locus of the function $F(P(y),\alpha,\beta) = \int_X \exp(\alpha)\cdot(1-P(\beta))$ and we have 
\[
\left.\frac{d}{dr}\right\vert_{r=0}F(P(y)-ry^k,\alpha,\beta)) = \frac{1}{(n-k)!}\int_X \alpha^{n-k}\cdot\beta^k > 0 
\]
where $k$ is any positive integer such that the coefficient of $y^k$ in $P(y)$ is non-zero. Such a $k$ always exists by assumption.  
Then, for any choice of Kähler forms $\omega\in\alpha,\theta\in\beta$ the triple $(P(y),\alpha,\beta)$ corresponds to the gMA equation that seeks a smooth $\psi\in\mathcal H(\omega)$ such that 
\begin{equation}\label{eqn:gMAcoh2}
    \exp(\omega_\psi)^{[n,n]}=(\exp(\omega_\psi)\wedge P(\theta))^{[n,n]},
\end{equation}
where $P(\theta)$ is the multi-degree form given by substituting $\theta$ for $y$ in the polynomial $P(y)$. Let us denote by $\mathcal S(X,\mathrm{gMA})$ the set of triples $(P(y),\alpha,\beta)\in\mathbb R_+[y]_n\times\mathcal K_X \times \mathcal K_X$ that satisfy the cohomological condition \eqref{eqn:gMAcoh2}. Recall that by Proposition \ref{lem:gMAfactorisable}, each triple $(P(y),\alpha,\beta)\in\mathcal S(X,\mathrm{gMA})$ defines a factorisable gMA equation. Let us denote by $\tau_p(\alpha,\beta,P(y))$ the associated factor class of degree $p$, as given by Definition \ref{def:factorisablegMAeq}. We shall denote by $\mathcal S_+(X,\mathrm{gMA})$ the subset of those triples $(P(y),\alpha,\beta)$ such that the associated factor classes $\tau_p(\alpha,\beta,P(y))$ of degree $p$ are $(p+1)$-modified Kähler classes. Our aim is to prove a wall structure type statement for the set $\mathcal S_+(X,\mathrm{gMA})$. 

\begin{lemma}
    For each $p=1,\dots,\dim_\mathbb C X -1$, the assignment 
    \[
    T:\mathcal S(X,\mathrm{gMA})\to (H^{1,1}(X,\mathbb R))^{n-1},\quad (P(y),\alpha,\beta)\mapsto (\tau_p(\alpha,\beta,P(y)))_{p=1}^{n-1}
    \]
    is continuous. 
\end{lemma}
\begin{proof}
    Recall that $\tau_p(\alpha,\beta,P(y)) = \alpha - r_p\beta$, where $x-r_py$ is the linear form given by the factorisation 
    \[
    Q_p(x,y) = (x-r_py) \tilde Q_p(x,y).
    \]
    Here $Q_p(x,y)$ is the degree $p$ homogeneous part of the power series $\exp(x)(1-P(y))$, and $\tilde Q_p(x,y)$ has non-negative coefficients. By the proof of Lemma \ref{lem:gMAfactorisable}, it follows that $r_p$ is uniquely determined as the largest real root of $h_p(x) = Q_p(x,1)$. But the roots of a polynomial depend continuously on the coefficients.  
\end{proof}

\begin{corollary}
    The set $\mathcal S_+(X,\mathrm{gMA})$ is open in the manifold with corners $\mathcal S(X,\mathrm{gMA})$.
\end{corollary}
\begin{proof}
    This follows immediately from the Lemma and the fact that $\mathcal M_p \mathcal K_X$ are open (by definition) in $H^{1,1}(X,\mathbb R)$. 
\end{proof}

Now, let $(P(y),\alpha,\beta)\in \mathcal S_+(X,{\mathrm{gMA}})$ be given. Using Lemma \ref{lem:wallchamberlemma}, we can find open neighbourhoods $U_p\subseteq \mathcal M_{p+1}\mathcal K_X$ of $\tau_p(\alpha,\beta,P(y))$ and finite sets $S_p$ of $p$-dimensional irreducible
subvarieties of $X$ such that for all $\tau^\prime \in U_p$, all irreducible $p$-dimensional subvarieties $V$ and all K\"ahler forms $\omega_1,\dots,\omega_{p-1}$ whenever we have 
\[\int_V \tau_p(\alpha,\beta,P(y))\cdot[\omega_1]\cdot\dots\cdot[\omega_{p-1}] \leq 0\]
then $V \in S_p$. Let $U =  T^{-1}(U_1\times U_2 \times\dots\times U_{n-1})$. Then $U$ is an open neighbourhood of $(P(y),\alpha,\beta)$ in $\mathcal S_+(X,{\mathrm{gMA}})$. In fact, for any Kähler form $\theta\in\beta$, if $V$ is any (gMA) destabiliser for the triple $(X,\alpha,P(\theta))$ then we have 
\[
\int_V \exp(\alpha)\cdot(1-P(\beta)) = \int_V \tau_p(\alpha,\beta,P(y))\cdot \tilde Q_p(\alpha,\beta)) \leq 0,
\]
but since $\tilde Q_p(\alpha,\beta)$ is a non-negative linear combination of products of powers of Kähler classes $\alpha$ and $\beta$, we must have 
\[
\int_V \tau_p(\alpha,\beta)\cdot \alpha^r \cdot \beta^{p-r-1} \leq 0
\]
for some $0\leq r \leq p-1$. But since $\tau_p(\alpha,\beta,P(y))\in U_p$ we get immediately that this implies that $V\in S_p$. 

Let us denote by $\mathcal S_+(X,\mathrm{gMA})^\mathrm{Stab}$ the locus of those $(P(y),\alpha,\beta)$ such that $(X,\alpha,P(\theta))$ is (gMA) stable for any choice of Kähler form $\theta \in \beta$. 

\begin{theorem}
The boundary $\partial\mathcal S_+(X,\mathrm{gMA})^{\mathrm{Stab}}$ of $\mathcal S_+(X,\mathrm{gMA})^{\mathrm{Stab}}$ in $\mathcal S_+(X,\mathrm{gMA})$  is a locally finite union of closed submanifolds $W$ of $\mathcal S_+(X,\mathrm{gMA})$ of (real) codimension one, each one of them cut out by an equation of the form
        \[
        \int_V \exp(\alpha)\cdot(1-P(\beta)) = 0.
        \] 
\end{theorem}
\begin{proof}
    The local finiteness follows from the discussion preceding the statement of the Theorem. The only claim that needs justification is that the boundary loci are submanifolds of codimension one. But they are the zero locus of a function $F$ of the form $F(P(y),\alpha,\beta))= \int_V \exp(\alpha)\cdot(1-P(\beta))$. One can easily show that zero is a regular value of this function.
\end{proof}
\bigskip


\end{document}